\documentclass[12pt, reqno]{amsart}

\usepackage[margin=1in]{geometry}
\usepackage{amsfonts,appendix, latexsym, amssymb, amsmath, bbm, fullpage, amsthm}
\usepackage[alphabetic]{amsrefs}
\usepackage{amsmath}
\usepackage[subnum]{cases}
\usepackage[colorlinks, citecolor=blue]{hyperref}
\numberwithin{equation}{section}
\newcommand{\e}{\varepsilon}

\newcommand{\p}{\mathbf p}
\newcommand{\R}{\mathbb R}
\newcommand{\N}{\mathbb N}

\newcommand{\E}{\mathbb E}
\newcommand{\Pro}{\mathbb P}
\newcommand{\vr}{\mathrm{vr}}
\newcommand{\vol}{\mathrm{vol}}
\newcommand{\tra}{\mathrm{trace}}
\newcommand{\id}{\mathrm{id}}
\newcommand{\norm}[1]{\left \lVert#1 \right\rVert}
\newcommand{\abs}[1]{\left\lvert#1 \right\rvert}

\newcommand{\X}{X_{\pi}^n}

\theoremstyle{plain}
\newtheorem{thm}{Theorem}%[section]
\newtheorem*{thm*}{Theorem}
\newtheorem{thmalpha}{Theorem}

\newtheorem{cor}{Corollary}[section]
\newtheorem{lem}[cor]{Lemma}
\newtheorem{prop}[cor]{Proposition}

\theoremstyle{definition}

\newtheorem{rmk}[cor]{Remark}

\begin{document}

%%%%%%%%%%%%%%%%%%%%%%%%%%%%%%%%%%%%%%%%%%%%%5

\title{On the geometry of projective tensor products}

\author[O. Giladi]{Ohad Giladi}
\address{School of Mathematical and Physical Sciences, University of Newcastle, Callaghan, NSW 2308, Australia}
\email{ohad.giladi@newcastle.edu.au}
\thanks{O.G. was supported in part by the Australian Research Council}

\author[J. Prochno]{Joscha Prochno}
\address{School of Mathematics and Physical Sciences, University of Hull, Cottingham Road, Hull, HU6 7RX, United Kingdom}
\email{j.prochno@hull.ac.uk}
\thanks{J.P. was supported in part by the FWF grant FWFM 162800.}

\author[C. Sch\"utt]{Carsten Sch\"utt}
\address{Mathematisches Seminar, Christian-Albrechts-University Kiel, Ludewig-Meyn-Stra\ss e 4, 24098 Kiel, Germany}
\email{schuett@math.uni-kiel.de}

\author[N. Tomczak-Jaegermann]{Nicole Tomczak-Jaegermann}
\address{Department of Mathematical and Statistical Sciences, University of Alberta, 632 Central Academic Building, Edmonton, AB T6G 2G1, Canada}
\email{nicole.tomczak@ualberta.ca}
\thanks{A part of this work was done when N.T-J. held the
Canada Research Chair in Geometric Analysis; also supported by NSERC
Discovery Grant.}

\author[E. Werner]{Elisabeth Werner}
\address{Department of Mathematics, Case Western Reserve University, 10900 Euclid Avenue, Cleveland, OH 44106, USA}
\email{elisabeth.werner@case.edu}
\thanks{E.W. was supported by NSF grant DMS-1504701}

\date{\today}

\begin{abstract}
In this work, we study the volume ratio of the projective tensor products $\ell^n_p\otimes_{\pi}\ell_q^n\otimes_{\pi}\ell_r^n$ with $1\leq p\leq q \leq r \leq \infty$. We obtain asymptotic formulas that  are sharp in almost all cases. As a consequence of our estimates, these spaces allow for a nearly Euclidean decomposition of Kashin type whenever $1\leq p \leq q\leq r \leq 2$ or $1\leq p \leq 2 \leq r \leq \infty$ and $q=2$. Also, from the Bourgain-Milman bound on the volume ratio of Banach spaces in terms of their cotype $2$ constant, we obtain information on the cotype of these $3$-fold projective tensor products. Our results naturally generalize to  $k$-fold products $\ell_{p_1}^n\otimes_{\pi}\dots \otimes_{\pi}\ell_{p_k}^n$ with $k\in\N$ and $1\leq p_1 \leq \dots\leq p_k \leq \infty$.
\end{abstract}

\subjclass[2010]{46A32, 46B28}

\maketitle
%\tableofcontents

%%%%%%%%%%%%%%%%%%%%%%%%%%%%%%%%%
\section{Introduction}
%%%%%%%%%%%%%%%%%%%%%%%%%%%%%%%%%

In the geometry of Banach spaces the volume ratio $\vr(X)$ of an $n$-dimensional normed space $X$ is defined as the $n$-th root of the volume of the unit ball in $X$ divided by the volume of its John ellipsoid. This notion plays an important role in the local theory of Banach spaces and has significant applications in approximation theory. It formally originates in the works \cite{Sza78} and \cite{STJ80}, which was influenced by the famous paper of B. Kashin \cite{Kas77} on nearly Euclidean orthogonal decompositions. Kashin discovered that for arbitrary $n\in\N$, the space $\ell_1^{2n}$ contains two orthogonal subspaces which are nearly Euclidean, meaning that their Banach-Mazur distance to $\ell_2^n$ is bounded by an absolute constant. S. Szarek \cite{Sza78}  noticed that the proof of this result  depends solely on the fact that $\ell_1^n$ has a bounded volume ratio with respect to $\ell_2^n$. In fact, it is essentially contained in the work of Szarek that if $X$ is a $2n$-dimensional Banach space, then there exist two $n$-dimensional subspaces each having a Banach-Mazur distance to $\ell_2^n$ bounded by a constant times the volume ratio of $X$ squared.  This observation by S. Szarek and N. Tomczak-Jaegermann was  further investigated in \cite{STJ80}, where the concept of volume ratio was formally introduced, its connection to the cotype 2 constant of Banach spaces was studied, and Kashin type decompositions were proved for some classes of Banach spaces, such as the projective tensor product spaces $\ell_p^n\otimes_{\pi}\ell_2^n$, $1\leq p \leq 2$.

Given two vector spaces $X$ and $Y$, their algebraic tensor product $X\otimes Y$ is the subspace of the dual space of all bilinear maps on $X\times Y$ spanned by elementary tensors $x\otimes y$, $x\in X$, $y\in Y$ (a formal definition is provided below). The theory of tensor products was established by A. Grothendieck in $1953$ in his R\'esum\'e \cite{Gro53} and has a huge impact on Banach space theory (see, e.g.,  the survey paper~\cite{Pis12}). This impact and the success of the concept of tensor products is to a large extent due to the work \cite{LP68} of J. Lindenstrauss and A. Pe\l czynski in the late sixties who reformulated Grothendieck's ideas in the context of operator ideals and made this theory accessible to a broader audience. Today, tensor products appear naturally in numerous applications, among others, in the entanglement of qubits in quantum computing, in quantum information theory in terms of (random) quantum channels (e.g., \cite{ASW10, ASW11, SWZ11} or in theoretical computer science to represent locally decodable codes \cite{E09}. For an interesting and recently discovered connection between the latter and the geometry of Banach spaces we refer the reader to \cite{BNR12}. %(\textcolor{red}{JP: In this last part we need more references!}) OG: seems perfectly sufficient to me

%We do not have to look far to see that 
The geometry of tensor products of Banach spaces is complicated, 
%highly non-trivial, 
even if the spaces involved are of simple geometric structure. For example, the $2$-fold projective tensor product of Hilbert spaces, $\ell_2\otimes_\pi\ell_2$, is naturally identified with the Schatten trace class $S_1$, the space of all compact operators $T:\ell_2\to\ell_2$ equipped with the norm $\|T\|_{S_1}=\tra\big(\sqrt{T^*T}\big)$. 
This space does not have local unconditional structure \cite{GL74}. 
%It is therefore hardly surprising that 
%from a certain perspective, 
%very little is known about geometric properties of 
%these spaces. 
The geometric structure of triple tensor products is even more complicated  and therefore it is hardly surprising that 
very little is known about the geometric properties of these spaces. 
For instance, regarding the permanence of cotype (see below for the definition) under projective tensor products,  it was proved by N. Tomczak-Jaegermann in \cite{T74} that the space $\ell_2\otimes_\pi\ell_2$ has cotype $2$, but the corresponding question in the $3$-fold case is still open for more than $40$ years.
% years now. 
 G. Pisier proved in \cite{P90} and \cite{P92} that the space $L_p\otimes_\pi L_q$ has cotype $\max\left(p,q\right)$ if $p,q\in[2,\infty)$ and, till the present day,  it is unknown whether these spaces have a non-trivial cotype when $p\in(1,2)$ and $q\in(1,2]$. In the recent paper \cite{BNR12} by J. Bri\"et, A. Naor and O. Regev they showed that the spaces $\ell_p\otimes_{\pi}\ell_q\otimes_{\pi}\ell_r$ fail to have non-trivial cotype if $\frac{1}{p}+\frac{1}{q}+\frac{1}{r} \leq 1$. Their proof uses  deep results from the theory of locally decodable codes. A direct, rather surprising consequence of their work is that for $p\in(1,\infty)$ the space $\ell_p \otimes_\pi \ell_{2p/(p-1)}\otimes_\pi \ell_{2p/(p-1)}$ fails to have non-trivial cotype, while, by Pisier's result, the $2$-fold projective tensor product $\ell_{2p/(p-1)}\otimes_\pi \ell_{2p/(p-1)}$ has finite cotype. 
% Before we proceed, l
Let us also mention that  interest in  cotype is, to a great deal, due to a famous result of B. Maurey and G. Pisier \cite{MP76}, who showed that a Banach space $X$ fails to have finite cotype if and only if it contains $\ell_\infty^n$'s uniformly.

 In view of the various open questions and surprising results around the geometry of projective tensor products, with the present paper, we contribute to a better understanding of the geometric structure of $3$-fold projective tensor products $\ell^n_p\otimes_{\pi}\ell_q^n\otimes_{\pi}\ell_r^n$ ($1\leq p\leq q \leq r \leq \infty$) by studying one of the key notions in local Banach space geometry,  the volume ratio of these spaces. This provides new structural insight and allows,  on the one hand,  to draw conclusions regarding  cotype properties of these spaces and,  on the other hand,  to see which of these spaces allow a nearly Euclidean decomposition of Kashin type.

\medskip
%%%%%%%%%%%%%%%%%%%%%%%%%%
\section{Presentation of the main result}\label{sec present}
%%%%%%%%%%%%%%%%%%%%%%%%%%

Given two Banach spaces, $(X,\|\cdot\|_X)$ and $(Y, \|\cdot\|_Y)$, the projective tensor product space, denoted $X\otimes_{\pi}Y$, is the space $X\otimes Y$ equipped with the norm
\begin{align}\label{def proj norm}
\|A\|_{X\otimes_{\pi}Y} = \inf\left\{\sum_{i=1}^m\|x_i\|_X\|y_i\|_Y~:~ A = \sum_{i=1}^m x_i\otimes y_i, ~x_i \in X,~ y_i \in Y \right\}.
\end{align}
See Section~\ref{sec def tensor} below for more information on projective tensor products.

Given an $n$-dimensional Banach space $(X,\|\cdot\|_X)$, let $B_X$ denote its unit ball, and let $\mathcal E_X$ be the ellipsoid of maximal volume contained in $B_X$. The volume ratio of $X$ is defined by
\begin{align}\label{def vr}
\vr(X) = \left(\frac{\vol_n(B_X)}{\vol_n(\mathcal E_X)}\right)^{1/n},
\end{align}
where $\vol_n(\cdot)$ denotes the $n$-dimensional Lebesgue measure.

Considering the importance of tensor products, it is natural to study the geometric properties of $X\otimes_{\pi}Y$ and, in particular,  the volume ratio $\vr(X\otimes_{\pi}Y)$ of these spaces. As already mentioned in the introduction, when $X = \ell_{p}^n$ ($1 \le p \le 2$) and $Y = \ell_2^n$ this was carried out  in \cite{STJ80}. The complete answer was given later by C. Sch\"utt in \cite{Sch82}, where it was proved that if $1\leq p \leq q \leq \infty$, then
        \begin{align*}
        \vr\left(\ell_{p}^n \otimes_{\pi} \ell_{q}^n\right) \asymp_{p,q}
        \begin{cases}
        1, & q \le 2.
        \\
        n^{\frac 1 2 - \frac 1 {q}}, & p\le 2 \le q,\, \frac 1{p}+\frac 1 {q} \ge 1,
        \\
        n^{\frac  1 {p}-\frac 1 2}, & p\le 2 \le q,\, \frac 1{p}+\frac 1 {q} \le 1,
        \\
        n^{\max\left(\frac 1 2 - \frac 1 {p} - \frac 1 {q}, 0 \right)}, & p\ge 2.
        \end{cases}
        \end{align*}

The notation $\asymp_{p,q}$ means equivalence up to constants that depend only on $p$ and $q$. In \cite{DM06} this was generalized to the setting $E\otimes_\pi F$, where $E$ and $F$ are symmetric Banach sequence spaces, each either $2$-convex or $2$-concave.

%The principal aim of the present paper is to estimate the volume ratio of $3$-fold projective tensor products $\ell^n_p\otimes_{\pi}\ell_q^n\otimes_{\pi}\ell_r^n$ where $1\leq p\leq q \leq r \leq \infty$. As a consequence, we obtain information on the cotype of $\ell_p\widehat{\otimes}_{\pi}\ell_q\widehat{\otimes}_{\pi}\ell_r$.

We study tensor products $\ell^n_p\otimes_{\pi}\ell_q^n\otimes_{\pi}\ell_r^n$ ($1\leq p \leq q\leq r\leq \infty$). Our main result is as follows.

\begin{thmalpha}\label{main thm}
Let $n\in\N$ and $1 \le p \le q \le r \le \infty$. Then we have
\begin{align*}
\vr\left(\ell^n_p\otimes_{\pi}\ell_q^n\otimes_{\pi}\ell_r^n \right) \asymp_{p,q,r}
\begin{cases}
1, & r \le 2,
\\
n^{\max\left(\frac 1 2 - \frac 1 q - \frac 1 r, 0\right)}, & p \le 2 \le q,\, \frac 1 p + \frac 1 q + \frac 1 r \ge 1,
\\
n^{\frac 1 p - \frac 1 2}, & p \le 2 \le q, \,\frac 1 p + \frac 1 q + \frac 1 r \le 1,
\\
n^{\max\left(\frac 1 2 - \frac 1 p - \frac 1 q - \frac 1 r, 0\right)}, & p \ge 2.
\end{cases}
\end{align*}
In the case $p \le q \le 2 \le r$, we have
\begin{align*}
1 \lesssim_{p,q,r} \vr\left(\ell^n_p\otimes_{\pi}\ell_q^n\otimes_{\pi}\ell_r^n \right) \lesssim_{p,q,r} n^{\min\left(\frac 1 2 - \frac 1 r, \frac 1 q - \frac 1 2\right)}.
\end{align*}
\end{thmalpha}

\medskip

Here and in what follows $\lesssim_{p,q,r}$, $\gtrsim_{p,q,r}$ mean inequalities with implied constants that depend only on the parameters $p,q,r$. The notation $\asymp_{p,q,r}$ means that we have both $\lesssim_{p,q,r}$ and $\gtrsim_{p,q,r}$.

\begin{rmk}\label{rem: special cases}
        We would like to remark that whenever $1\leq p\leq q \leq 2 \leq r \leq \infty$, we are able to improve on the general bound given by Theorem \ref{main thm} in the following situations (see Corollary \ref{cor:improvement open case}):
        \begin{align*}
        \vr\left(\ell^n_p\otimes_{\pi}\ell_q^n\otimes_{\pi}\ell_r^n \right) \asymp_{p,q,r}
        \begin{cases}
        n^{\min\left(\frac{1}{q}-\frac{1}{2},\frac{1}{2}-\frac{1}{r}\right)}, &  p=1 \leq q \leq 2 \leq r, \\
        n^{\frac{1}{q}-\frac{1}{2}}, & p = q \leq 2,\, r=\infty,\\
        1, &  1\leq p \leq q =2 \leq r.
        \end{cases}
        \end{align*}        
\end{rmk}

\begin{rmk}
 Theorem \ref{main thm} and Remark \ref{rem: special cases} immediately imply that the spaces $\ell^n_p\otimes_{\pi}\ell_q^n\otimes_{\pi}\ell_r^n$ allow a nearly Euclidean decomposition of Kashin type, when $1\leq p \leq q \leq r \leq 2$ or $1\leq p \leq 2 \leq r \leq \infty$ and $q=2$.
\end{rmk}

Before we proceed, let us comment on the strategy of the proof. Recall first that a Banach space $(X,\|\cdot\|_X)$ is said to have enough symmetries if the only operators that commute with every isometry on $X$ are multiples of the identity. It is known that if $(X,\|\cdot\|_X)$ is $n$-dimensional and has enough symmetries, then $\mathcal E_X$ is given by
\begin{align*}
\mathcal E_X = \left\|\id: \ell_2^n \to X\right\|^{-1} B_2^n,
\end{align*}
where $B_2^n$ denotes the $n$-dimensional Euclidean ball (see, for instance,~\cite[Sec.~16]{TJ89}). Hence, using formula~\eqref{def vr}, if $(X,\|\cdot\|_X)$ is $n$-dimensional and has enough symmetries, then
\begin{align}\label{formula vol ell}
\vr(X) =  \left(\frac{\vol_n(B_X)}{\vol_n(B_2^n)}\right)^{1/n} \big\|\id:\ell_2^n\to X\big\|  \stackrel{(*)}{\asymp} \sqrt n\big(\vol_n(B_X)\big)^{1/n} \big\|\id:\ell_2^n\to X\big\|,
\end{align}
where in ($*$) we used Stirling's formula to deduce $\vol(B_2^n)^{1/n} \asymp 1/\sqrt n$. It is also known that projective tensor products of $\ell_p$ spaces are spaces with enough symmetries and therefore formula \eqref{formula vol ell} holds for the spaces $\ell^n_p\otimes_{\pi}\ell_q^n\otimes_{\pi}\ell_r^n$. Thus, in order to prove Theorem \ref{main thm}, it is enough to compute the volume of the unit ball of $\ell^n_p\otimes_{\pi}\ell_q^n\otimes_{\pi}\ell_r^n$, and the norm of the natural identity between $\ell_2^{n^3}$ and $\ell^n_p\otimes_{\pi}\ell_q^n\otimes_{\pi}\ell_r^n$.

  The rest of this paper is organized as follows. In Section~\ref{sec prelim} we collect some basic results which will be used later. In Section~\ref{sec vol} we estimate the volume of the unit ball in $\ell_p^n \otimes_{\pi} \ell_q^n \otimes_{\pi} \ell_r^n$. In Section~\ref{sec norm} we estimate the norm of the natural identity. In Section~\ref{sec k fold} we discuss the case of $k$-fold tensor products. Finally, in Section~\ref{sec applications}, we present some applications of Theorem~\ref{main thm}.

\medskip
%%%%%%%%%%%%%%%%%%%%%%%%%%%%%%
\section{Preliminaries}\label{sec prelim}
%%%%%%%%%%%%%%%%%%%%%%%%%%%%%%

In this section, we introduce the necessary notions and background material
% that we use throughout this note 
and provide the main ingredients needed to prove the estimates for the volume ratio of $3$-fold projective tensor products. These include an extension of Chevet's inequality, a lower bound on the volume of unit balls in Banach spaces due to Sch\"utt, the famous Blaschke-Santal\'o inequality, and a multilinear version of an inequality of Hardy and Littlewood.

%%\medskip

%%%%%%%%%%%%%%%%%%
\subsection{General notation}
%%%%%%%%%%%%%%%%%%

Given a Banach space $(X, \|\cdot\|_X)$, denote its unit ball by $B_X$ and its dual space by $X^*$. For two Banach spaces $X$ and $Y$, we write $\mathcal L(X,Y)$ for the space of all bounded linear operators from $X$ to $Y$.

For $1 \leq p \leq \infty$,  $\ell_p^n$ is  the vector space $\R^n$  with the norm
\begin{align*}
\|(x_i)_{i=1}^n\|_p =
\begin{cases}
\big(\sum_{i=1}^n |x_i|^p\big)^{1/p}\,, & 1\leq p <\infty \\
\max\limits_{1 \leq i \leq n}|x_i|\,,& p=\infty.
\end{cases}
\end{align*}
The unit ball in $\ell_p^n$ is denoted by $B_p^n=\{ x\in\R^n \,:\, \|x\|_p\leq 1\}$. 
%By $p^*$ we will always denote 
The conjugate $p^*$ of $p$  is defined via the relation $\frac{1}{p}+\frac{1}{p^*}=1$. Unless otherwise stated, $e_1,\dots,e_n$, will  be the standard unit vectors in $\R^n$.

We shall also use the asymptotic notations $\lesssim$ and $\gtrsim$ to indicate the corresponding inequalities up to universal constant factors, and we shall denote equivalence up to universal constant factors by $\asymp$, where $ A \asymp B$ is the same as $(A \lesssim B) \land (A \gtrsim B)$. If the constants involved depend on a parameter $\alpha$, we denote this by $\lesssim_\alpha$, $\gtrsim_\alpha$ and $\asymp_\alpha$, respectively.

%%\medskip

%%%%%%%%%%%%%%%%%%%%%%%%%%%%
\subsection{Polar body and Blaschke-Santal\'o inequality}
%%%%%%%%%%%%%%%%%%%%%%%%%%%%

A convex body $K$ in $\R^n$ is a compact convex subset of $\R^n$ with non-empty interior. The $n$-dimensional Lebesgue measure of a convex body $K\subseteq \R^n$ is denoted by $\vol_n(K)$. If $0$ is an interior point of a convex body $K$ in $\R^n$, we define the polar body of $K$ by
\[
K^\circ:=\big\{ y\in\R^n \,:\, \forall x \in K:\, \langle x,y \rangle \leq 1 \big\},
\]
where $\langle \cdot, \cdot\rangle$ stands for the standard inner product on $\R^n$. Note that the unit ball of any norm on $\R^n$ is a convex body and its polar body is just the unit ball of the corresponding dual norm. Moreover, we have $(K^\circ)^\circ=K$. The famous Blaschke-Santal\'o inequality provides a sharp upper bound for the volume product of a convex body with its polar (see, for example,~\cite[Sec.~7]{Pis89} or \cite{Schn14}).
\begin{lem}[Blaschke-Santal\'o inequality]\label{lem:blaschke santalo}
        Let $K$ be an origin symmetric convex body in $\R^n$. Then
        \[
        \vol_n(K)\cdot\vol_n(K^\circ) \leq \vol_n(B_2^n)^2,
        \]
        with equality if and only if $K$ is an ellipsoid.
\end{lem}

%\medskip

%%%%%%%%%%%%%%%%%%%%%%%%%%%%%%%
\subsection{Tensor products and extended Chevet inequality}\label{sec def tensor}
%%%%%%%%%%%%%%%%%%%%%%%%%%%%%%%

Given two Banach spaces $(X,\|\cdot\|_X)$ and $(Y, \|\cdot\|_Y)$, the algebraic tensor product $X\otimes Y$ can be constructed as the space of linear functionals on the space of all bilinear forms on $X\times Y$. Given $x\in X$, $y\in Y$ and a bilinear form $B$ on $X\times Y$, we define $(x\otimes y)(B) := B(x,y)$. On the tensor product space define the projective tensor product space, denoted by $X\otimes_\pi Y$, as $X\otimes Y$ equipped with the norm
\begin{align*}
\|A\|_{X\otimes_{\pi}Y} := \inf\left\{\sum_{i=1}^m\|x_i\|_X\|y_i\|_Y \,:\, A = \sum_{i=1}^mx_i\otimes y_i, ~ x_i \in X, ~ y_i \in Y\right\}.
\end{align*}
Also, define the injective tensor norm space, denoted $X\otimes_{\epsilon}Y$, as the tensor product space $X\otimes Y$ equipped with the norm
\begin{align*}
\|A\|_{X\otimes_{\epsilon}Y} := \sup\left\{\Big|\sum_{i=1}^n\varphi(x_i)\psi(y_i) \Big|\,:\, A = \sum_{i=1}^m x_i \otimes y_i,~ \varphi\in B_{X^*}, ~\psi\in B_{Y^*}\right\}.
\end{align*}
Here, we only consider finite-dimensional spaces and therefore the tensor products are always complete. In this case, it can be shown that $X\otimes_\pi Y = \mathcal N(X^*,Y)$, the space of all nuclear operators from $X^*$ into $Y$. We will often use the fact that $\left(X\otimes_{\pi}Y\right)^*$, the dual space of $X\otimes_{\pi}Y$, is the space of operators from $X^*$ to $Y$,  $\mathcal L(X^*,Y)$, equipped with the standard operator norm. It is also known that $\mathcal L(X^*,Y)$ can be identified with the injective tensor product $X\otimes_{\epsilon}Y$. In particular, for $A\in X\otimes_{\epsilon}Y$, we have
\begin{align*}
\|A\|_{X\otimes_{\epsilon}Y} = \sup_{x\in B_{X^*}}\|Ax\|_Y.
\end{align*}
We refer the reader to \cite{Rya02} and~\cite{DFS08} for more information about tensor products. %\textcolor{red}{({\bf JP}: add reference for the book ``The Metric Theory of Tensor Products: Grothendieck's Resume Revisited'' by Diestel et. al.)}

Recall that for a sequence $x=(x_i)_{i=1}^n$ in a Banach space $(X,\|\cdot\|_X)$, the norm $\|x\|_{\omega,2}$ is given by
$$
\norm{x}_{w,2} := \sup_{\norm{\varphi^*}_{X^*}=1} \left( \sum_{i=1}^n \abs{\varphi^*(x_i)}^2 \right)^{\frac{1}{2}}.
$$

The following inequality is due to Chevet \cite{Che78}.

\begin{lem}[Chevet's inequality]
        Let $(X,\|\cdot\|_X)$, $(Y,\|\cdot\|_Y)$ be two Banach spaces and consider  sequences $x_1,\dots,x_m\in X$, $y_1,\dots,y_n\in Y$, and sequences $(g_{i,j})_{i,j=1}^{m,n}$, $(\xi_i)_{i=1}^m$, $(\eta_j )_{j=1}^n$ of independent identically distributed standard Gaussians random variables. Then
        \begin{align}\label{ineq chevet 2}
        \mathbb E\,\bigg\|\sum_{i, j=1}^{m,n} g_{i,j} x_{i} \otimes y_{j} \bigg\|_{X\otimes_{\epsilon}Y}
        & \le \norm{(x_{i})_{i=1}^m}_{\omega,2}\mathbb E\,\bigg\|\sum_{j=1}^n\eta_{j}y_{j}\bigg\|_{Y}+\norm{(y_{j})_{j=1}^n}_{\omega,2}\mathbb E\,\bigg\|\sum_{i=1}^m\xi_{i}x_{i}\bigg\|_{X}.
        \end{align}
\end{lem}

It is known that $B_{X^*\otimes_{\pi} Y^*} = \mathrm{conv}\left(B_{X^*}\otimes B_{Y^*}\right)$ (see, for example, Proposition 2.2 in \cite{Rya02}) and thus
\begin{align}\label{weak norm prod}
\big\|(x_{i}\otimes y_{j})_{i,j}\big\|_{\omega,2} = \big\|(x_{i})_{i}\big\|_{\omega,2} \cdot \big\|(y_{j})_{j}\big\|_{\omega,2}.
\end{align}
Using inequalities~\eqref{ineq chevet 2} and~\eqref{weak norm prod}, we obtain the following $3$-fold version of Chevet's inequality.

\begin{lem}[$3$-fold Chevet inequality]\label{cor chevet}
        Let $(X,\|\cdot\|_X)$, $(Y,\|\cdot\|_Y)$, $(Z,\|\cdot\|_Z)$ be Banach spaces. Assume that $x_1,\dots, x_m \in X$, $y_1, \dots, y_n \in Y$ and $z_1, \dots, z_{\ell} \in Z$.  Let $g_{i,j,k}$, $\xi_i$, $\eta_j$, $\rho_k$,  $i = 1,\dots,m$, $j=1,\dots,n$, $k=1,\dots,\ell$, be independent standard Gaussians random variables. Then
        \[
        \E\, \bigg\|\sum_{i,j,k=1}^{m,n,\ell} g_{ijk} \,x_i \otimes y_j \otimes z_k\bigg\|_{X\otimes_{\epsilon}Y\otimes_{\epsilon}Z}
        \leq  \Lambda,
        \]
        where
        \begin{eqnarray*}
                \Lambda
                & := & \norm{(x_i)_{i=1}^m}_{w,2}\norm{(y_j)_{j=1}^n}_{w,2}\E\,\bigg\|\sum_{k=1}^{\ell}\rho_kz_k\bigg\|_Z
                + \norm{(x_i)_{i=1}^m}_{w,2}\norm{(z_k)_{k=1}^\ell}_{w,2}\E\,\bigg\|\sum_{j=1}^{n}\eta_jy_j\bigg\|_Y \\
                && + \norm{(y_j)_{j=1}^n}_{w,2}\norm{(z_k)_{k=1}^\ell}_{w,2}\E\,\bigg\|\sum_{i=1}^{m}\xi_ix_i\bigg\|_X .
        \end{eqnarray*}
\end{lem}

We would like to point out that, simply using the triangle inequality, a corresponding lower bound can be obtained up to an absolute constant.

%\medskip

%%%%%%%%%%%%%%%%%%%%%%%%%
\subsection{Volume ratio and Rademacher cotype }
%%%%%%%%%%%%%%%%%%%%%%%%%

The concept of Rademacher cotype was introduced to Banach space theory by J. Hoffmann-J\o rgensen~\cite{HJ74} in the early $1970$s. The basic theory was developed by B. Maurey and G. Pisier \cite{MP76}.

A Banach space $(X,\|\cdot\|_X)$, is said to have Rademacher cotype $\alpha$ if there exists a constant $C\in(0,\infty)$ such that for all $m\in \mathbb N$ and all $x_1,\dots,x_m\in X$,
\begin{align}\label{ineq ecotype}
\Big(\sum_{i=1}^m\|x_i\|_X^\alpha\Big)^{1/\alpha} \le C\, \mathbb E\,\Big\|\sum_{i=1}^m\e_ix_i\Big\|_X.
\end{align}
In~\eqref{ineq ecotype} and in what follows, $(\varepsilon_i)_{i=1}^\infty$ denotes a sequence of independent symmetric Bernoulli random variables, that is, $\Pro(\varepsilon_i=1)=\Pro(\varepsilon_i=-1)=\frac{1}{2}$ for all $i=1,\dots,n$. The smallest $C$ that satisfies \eqref{ineq ecotype} is denoted by $C_\alpha(X)$ and called the cotype $\alpha$ constant of $X$. By taking $x_1=x_2=\dots=x_m$ it follows that necessarily $\alpha \ge 2$.

In \cite{BNR12} it was shown that, whenever $\frac{1}{p} + \frac{1}{q} + \frac{1}{r} \leq 1$, then
\begin{align*}
C_\alpha\left(\ell_{p}\otimes_{\pi}\ell_{q}\otimes_{\pi}\ell_{r}\right) = \infty,
\end{align*}
for every $2\le \alpha < \infty$. However, it is still an open question whether for example $\ell_2 \otimes_\pi \ell_2 \otimes_\pi \ell_2$ has finite cotype.

One reason to study  volume ratio of tensor products is its relation to the cotype constant. More precisely, given an estimate on volume ratio, one can use the following result 
by Bourgain and Milman~\cite{BM87}) which connects volume ratio and cotype of a Banach space.

\begin{thm}[\cite{BM87}]\label{thm BM}
        Let $X$ be a Banach space. Then
        \begin{align*}
        \vr(X) \lesssim C_2(X)\log \left(2C_2(X)\right).
        \end{align*}
\end{thm}

The relation between volume ratio and cotype property is far from being well understood. For example, it was asked in~\cite{STJ80} whether bounded volume ratio implies cotype $q$ for every $q>2$. 

In Section~\ref{sec applications} below we discuss the cotype property of the space $\ell_p^n \otimes_{\pi} \ell_q^n \otimes_{\pi} \ell_r^n$, as well as the case of $k$-fold tensor products.

%\medskip

%%%%%%%%%%%%%%%%%%%%%%%%%%%%%%%%%%
\subsection{Volume of unit balls in Banach spaces}
%%%%%%%%%%%%%%%%%%%%%%%%%%%%%%%%%%

The following result is a special case of Lemma 1.5 in~\cite{Sch82}. It provides a lower estimate for the volume of the unit ball of a normed space by the volume of a $B_\infty^n$ ball of a certain radius, arising from an average over sign vectors.

\begin{lem}\label{LEM_volume_unit_balls}
        Let $(X,\|\cdot\|_X)$ be an $n$-dimensional Banach space, and let $e_1,\dots,e_n$ be basis vectors such that $\|e_i\|_X=1$, and $(\varepsilon_i)_{i=1}^\infty$ a sequence of independent symmetric Bernoulli random variables. Then
        $$
        2^n \Bigg( \mathbb E\,\Big\|\sum_{i=1}^n \varepsilon_ie_i\Big\|_X\Bigg)^{-n} \leq \vol_n(B_X).
        $$
\end{lem}

We will use this result in combination with the $3$-fold version of Chevet's inequality to obtain a lower bound on the volume of the unit ball in $\ell^n_{p^*}\otimes_{\epsilon}\ell_{q^*}^n\otimes_{\epsilon}\ell_{r^*}^n$ which gives, using the Blaschke-Santal\'o inequality, an upper bound on the volume of the unit ball in the dual space $\ell_p\otimes_{\pi}\ell_q\otimes_{\pi}\ell_r$, as well as tensor products of more than three $\ell_p$ spaces.

%\medskip
%%%%%%%%%%%%%%%%%%%%%%%%%%%
\subsection{Rademacher versus Gaussian averages}
%%%%%%%%%%%%%%%%%%%%%%%%%%%

In order to use Chevet's inequality in combination with Lemma~\ref{LEM_volume_unit_balls}, we need to pass from a Rademacher average to a Gaussian one. The following result due to Pisier shows that Rademacher averages are dominated by Gaussian averages in arbitrary Banach spaces. Note however that in general these averages are not equivalent.

\begin{lem}[\cite{Pis86}]\label{lem:gauss dominates rademacher}
        Let $(X, \|\cdot\|_X)$ be a Banach space, $1\leq p <\infty$ and let $\xi_1,\dots,\xi_n$ be independent, symmetric random variables. Assume that $\E|\xi_i|=\E|\xi_j|$ for all $1\leq i,j \leq n$. Then, for all $x_1,\dots,x_n\in X$, we have
        \[
        \E \,\bigg\| \sum_{i=1}^n \varepsilon_i x_i\bigg\|^p \leq \big(\E|\xi_1|\big)^{-p}\, \E \,\bigg\| \sum_{i=1}^n \xi_i x_i\bigg\|^p.
        \]
\end{lem}

In particular, if we choose $g_1, \dots,g_n$ to be independent Gaussian random variables, then since  $\big(\E|g_1|\big)^{-p}=(\pi/2)^{p/2}$, we have
\begin{align*}
\E \,\bigg\| \sum_{i=1}^n \varepsilon_i x_i\bigg\|^p \lesssim_p \E \,\bigg\| \sum_{i=1}^n g_i x_i\bigg\|^p.
\end{align*}

%\medskip

%%%%%%%%%%%%%%%%%%%%%%%%%%%%%%%%%%
\subsection{A multilinear Hardy-Littlewood type inequality}
%%%%%%%%%%%%%%%%%%%%%%%%%%%%%%%%%%

An essential tool in proving upper bounds on the norm of the natural identity between $\ell_2^{n^3}$ and $\ell^n_p\otimes_{\pi}\ell_q^n\otimes_{\pi}\ell_r^n$ is the following inequality, which is a generalization of a classical inequality by Hardy and Littlewood \cite{HL34}. For now, we state it only for the case of $3$-fold tensors. In Section~\ref{sec k fold} we also present the consequences of the general version.

\begin{thm}[\cite{PP81}, Thm.~B]\label{thm multilinear}
        Let $n\in\N$ and $1\leq p,q,r \leq \infty$ so that $\frac{1}{p}+\frac{1}{q}+\frac{1}{r} \le \frac 1 2$. Then, for all $A\in \ell^n_{p^*}\otimes_{\epsilon}\ell_{q^*}^n\otimes_{\epsilon}\ell_{r^*}^n$, we have
        \begin{align*}
        \|A\|_{\ell_{\mu}^{n^3}} \lesssim \|A\|_{\ell^n_{p^*}\otimes_{\epsilon}\ell_{q^*}^n\otimes_{\epsilon}\ell_{r^*}^n},
        \end{align*}
        where $\mu$ is given by
        \begin{align}\label{def mu}
        \mu := \frac{3}{2-\frac{1}{p}-\frac{1}{q}-\frac{1}{r}}\,.
        \end{align}
\end{thm}

\medskip
%%%%%%%%%%%%%%%%%%%%%%%%%%%%%%%%%%%%%%%%%%%%%%%
\section{The volume of the unit ball in $\ell^n_p\otimes_{\pi}\ell_q^n\otimes_{\pi}\ell_r^n$}\label{sec vol}
%%%%%%%%%%%%%%%%%%%%%%%%%%%%%%%%%%%%%%%%%%%%%%%

In this section we evaluate the volume of the unit ball of $\ell^n_p\otimes_{\pi}\ell_q^n\otimes_{\pi}\ell_r^n$ up to constants depending only on the parameters $p,q,r$. The main ingredients in the proof are the $3$-fold version of Chevet's inequality (Lemma \ref{cor chevet}) and the Blaschke-Santal\'o inequality (Lemma \ref{lem:blaschke santalo}).

The next theorem will be the consequence of the following two subsections. 

\begin{thmalpha}\label{thm:volume unit balls projective}
        Let $n\in\N$ and $1\leq p \le q \le r \leq \infty$. Then
        \begin{align*}
        n^{-\min\left(\frac{1}{p},\frac1 2\right) -\min\left(\frac{1}{q},\frac1 2\right) -\frac{1}{r}-1} \leq \vol_{n^3}\big(B_{\ell^n_p\otimes_{\pi}\ell_q^n\otimes_{\pi}\ell_r^n}\big)^{1/n^3} \lesssim_{p,q,r}\, n^{-\min\left(\frac{1}{p},\frac1 2\right) -\min\left(\frac{1}{q},\frac1 2\right) -\frac{1}{r}-1}.
        \end{align*}
\end{thmalpha}

%\medskip

%%%%%%%%%%%%%%%%%%%%%%%
\subsection{Lower bounds on the volume}
%%%%%%%%%%%%%%%%%%%%%%%

Let us first we fix some more notation. For $1\leq p,q,r \leq \infty$ and $n\in\N$, let \[X^n_\pi := \ell^n_p\otimes_{\pi}\ell_q^n\otimes_{\pi}\ell_r^n,\] 
where we suppress the parameters $p,q,r$ that are clear from the context. For the corresponding norm, we write in short $\|\cdot\|_\pi$ instead of $\| \cdot\|_{X_\pi}$. Similarly, we define  $X_\epsilon^n: = \ell^n_{p^*}\otimes_{\epsilon}\ell_{q^*}^n\otimes_{\epsilon}\ell_{r^*}^n$ and write $\|\cdot\|_\epsilon$ instead of $\|\cdot\|_{X_\epsilon}$. We denote the corresponding unit balls by $B^n_{\pi}$ and $B_{\epsilon}^n$ respectively. Any $A\in X_\epsilon^n$, we express in the form
\begin{align*}
A = \sum_{i,j,k=1}^n A_{i,j,k}\,e_{i}\otimes e_{j}\otimes e_{k}.
\end{align*}

The main tool in proving a lower bound is the following estimate that compares the injective norm $\|\cdot\|_\epsilon$ with the $\ell_1$-norm in $\R^{n^3}$.

\begin{prop}\label{lem l1 norm}
Let $n\in\N$, $1\leq p ,  q , r \leq \infty$, and assume $A\in X^n_\epsilon$. Then
\begin{align*}
\|A\|_{\epsilon} \geq \frac{1}{2} \,n^{-\min\left(\frac{1}{p},\frac1 2\right) -\min\left(\frac{1}{q},\frac1 2\right) -\frac{1}{r}-1}\sum_{i,j,k=1}^n\left |A_{i,j,k}\right |.
\end{align*}
\end{prop}

\begin{proof}
To prove this, we identify $X^n_\epsilon$ with a space of operators. We have
\begin{align*}
\|A\|_{\epsilon} = \max_{\|x\|_{p} =\|y\|_{q}=\|z\|_r=1}\left|\sum_{i,j,k=1}^n A_{i,j,k}\,x_iy_jz_k\right|.
\end{align*}
In what follows, $\e,\delta,\eta \in \{-1,1\}^{n}$ denote sign vectors.
We divide the proof into three different cases, depending on which side of $2$ the parameters $p,q,r$ lie.

\noindent{\bf Case 1:} Assume that $2\leq p\leq q$. In this case we choose $x=n^{-\frac{1}{p}}(\e_1,\ldots,\e_n)$, $y:=n^{-\frac{1}{q}}(\delta_1,\ldots,\delta_n)$, and $z=n^{-\frac{1}{r}}(\eta_1,\ldots,\eta_n)$. Then we obtain

\begin{eqnarray*}
        \max_{\norm{x}_p =  \norm{y}_q=\norm{z}_r=1} \abs{\sum_{i,j,k=1}^n A_{i,j,k}x_i y_j z_k}
        & \ge  & n^{-\frac 1 p-\frac 1 q -\frac 1 r}\max_{\e,\delta,\eta\in\{-1,1\}^n} \abs{\sum_{i,j,k=1}^n A_{i,j,k}\,\e_i \delta_j  \eta_k } \cr
        & = &  n^{-\frac 1 p-\frac 1 q -\frac 1 r} \max_{\delta,\eta\in\{-1,1\}^n} \sum_{i=1}^n  \abs{\sum_{j,k=1}^n A_{i,j,k} \,\delta_j  \eta_k } \cr
        & \ge &  n^{-\frac 1 p-\frac 1 q -\frac 1 r}\max_{\eta\in\{-1,1\}^n}\sum_{i=1}^n  \frac{1}{2^{n}}\sum_{\delta\in\{-1,1\}^n} \abs{\sum_{j,k=1}^n A_{i,j,k}\,  \delta_j\eta_k }.
\end{eqnarray*}
Applying Khintchine's inequality and then H\"older's inequality, we obtain
\begin{eqnarray*}
        \max_{\eta\in\{-1,1\}^n}\sum_{i=1}^n  \frac{1}{2^{n}}\sum_{\delta\in\{-1,1\}^n} \abs{\sum_{j,k=1}^n A_{i,j,k} \, \delta_j\eta_k }
        & \ge &  \frac{1}{\sqrt 2}\max_{\eta\in\{-1,1\}^n}\sum_{i=1}^n  \bigg(\sum_{j=1}^n \Big|\sum_{k=1}^n A_{i,j,k}\,  \eta_k \Big|^2\bigg)^{1/2}\cr
        & \ge & \frac 1{\sqrt{2n}}\max_{\eta\in\{-1,1\}^n}\sum_{i=1}^n \sum_{j=1}^n \abs{\sum_{k=1}^n A_{i,j,k}\,  \eta_k }.
\end{eqnarray*}
Again, applying Khinchine's inequality and then H\"older's inequality,
\begin{align*}
         \max_{\eta\in\{-1,1\}^n}\sum_{i=1}^n \sum_{j=1}^n \abs{\sum_{k=1}^n A_{i,j,k}\,  \eta_k }
        & \ge   \sum_{i=1}^n \sum_{j=1}^n \frac{1}{2^n}\sum_{\eta\in\{-1,1\}^n}\abs{\sum_{k=1}^n A_{i,j,k}\,  \eta_k }\cr
        & \ge  \frac 1{\sqrt {2}}\sum_{i=1}^n \sum_{j=1}^n \left(\sum_{k=1}^n\abs{ A_{i,j,k} }^2\right)^{1/2}\cr
        & \ge  \frac 1{\sqrt{2n}}\sum_{i,j,k=1}^n |A_{i,j,k}|.
\end{align*}
%Therefore, in this case we have $\|A\|_{\epsilon} \ge n^{-\frac 1 p -\frac 1 q -\frac 1 r -1}$.

\noindent{\bf Case 2:} Assume that $p\le 2 \leq q$. In this case, we choose for $x\in B_p^n$ the standard unit vectors $e_1,\dots,e_n$ and $y,z$ as in the previous case. We get, again by using the inequalities of Khintchine and H\"older,
\begin{eqnarray*}
        \max_{\norm{x}_p =  \norm{y}_q=\norm{z}_r=1} \abs{\sum_{i,j,k=1}^n A_{i,j,k}\,x_i y_j z_k}
        & \ge  & n^{-\frac 1 q -\frac 1 r}\max_{1\le i\le n}\max_{\delta,\eta\in\{-1,1\}^n} \abs{\sum_{j,k=1}^n A_{i,j,k}\,\delta_j  \eta_k } \cr
        & =  & n^{-\frac 1 q -\frac 1 r}\max_{1\le i\le n}\max_{\eta\in\{-1,1\}^n} \sum_{j=1}^n\abs{\sum_{k=1}^n A_{i,j,k} \, \eta_k } \cr
        & \ge &  n^{-\frac 1 q -\frac 1 r-1}\max_{\eta\in\{-1,1\}^n} \sum_{i,j=1}^n\abs{\sum_{k=1}^n A_{i,j,k} \,\eta_k } \cr
        & \ge &  n^{-\frac 1 q -\frac 1 r-1}\sum_{i,j=1}^n\frac{1}{2^n}\sum_{\eta\in\{-1,1\}^n}\abs{\sum_{k=1}^n A_{i,j,k} \, \eta_k } \cr
        & \ge &  \frac{1}{\sqrt 2}n^{-\frac 1 q -\frac 1 r-1} \sum_{i,j=1}^n\left(\sum_{k=1}^n \abs{A_{i,j,k}}^2\right)^{1/2} \cr
        & \ge &  \frac{1}{\sqrt 2}n^{-\frac 1 q -\frac 1 r-\frac 3 2} \sum_{i,j,k=1}^n| A_{i,j,k}|.
\end{eqnarray*}

\noindent{\bf Case 3:} Assume that $p\le q \le 2$. Choose for $x\in B_p^n$ and $y\in B_q^n$ the standard unit vectors $e_1,\dots,e_n$ and $z$ as in the previous two cases. We have
\begin{eqnarray*}
        \max_{\norm{x}_p =  \norm{y}_q=\norm{z}_r=1} \abs{\sum_{j,k=1}^n A_{i,j,k}\,x_i y_j z_k}
        & \ge  & n^{-\frac 1 r}\max_{1\le i,j\le n}\max_{\eta\in\{-1,1\}^n} \abs{\sum_{k=1}^n A_{i,j,k}\, \eta_k } \cr
        & = & n^{-\frac 1 r}\max_{1\le i,j\le n}\sum_{k=1}^n\abs{A_{i,j,k}} \cr
        & \ge & n^{-\frac 1 r-2}\sum_{i,j,k=1}^n|A_{i,j,k}|.
\end{eqnarray*}
This completes the proof of the lemma.
\end{proof}

As an immediate consequence of Lemma \ref{lem l1 norm}, we obtain a lower bound on the volume radius of the unit ball $B^n_{\pi}$ in $X^n_\pi$.

\begin{cor}\label{prop lower bound ball}
Let $n\in\N$ and $1\leq p,  q , r \leq \infty$. Then, we have
\begin{align*}
\vol_{n^3}\big(B_{\pi}^n\big)^{1/n^3} \ge n^{-\min\left(\frac{1}{p},\frac1 2\right) -\min\left(\frac{1}{q},\frac1 2\right) -\frac{1}{r}-1}.
\end{align*}
\end{cor}

\begin{proof}
From Lemma \ref{lem l1 norm}, we obtain that
\begin{align*}
B_{\epsilon}^n \subseteq 2\,n^{\min\left(\frac{1}{p},\frac1 2\right) +\min\left(\frac{1}{q},\frac1 2\right) +\frac{1}{r}+1}  B_{1}^{n^3}.
\end{align*}
Switching to the polar bodies implies
\begin{align*}
B_{\pi}^n \supseteq \frac{1}{2}\,n^{-\min\left(\frac{1}{p},\frac1 2\right) -\min\left(\frac{1}{q},\frac1 2\right) -\frac{1}{r}-1} B_{\infty}^{n^3}.
\end{align*}
Taking volumes and the $n^3$-rd root, the previous inclusion immediately gives
\begin{align*}
\vol_{n^3}\big(B_{\pi}^n\big)^{1/n^3} \ge n^{-\min\left(\frac{1}{p},\frac1 2\right) -\min\left(\frac{1}{q},\frac1 2\right) -\frac{1}{r}-1},
\end{align*}
which completes the proof.
\end{proof}

%\medskip

%%%%%%%%%%%%%%%%%%%%%%%
\subsection{Upper bounds on the volume}
%%%%%%%%%%%%%%%%%%%%%%%

To compute the matching upper bound on the volume radius of $B_\pi^n$, we will use Lemma \ref{LEM_volume_unit_balls}. To be more precise, the idea is as follows. Using our extended version of Chevet's inequality (see Lemma \ref{cor chevet}), we obtain a lower bound for the volume of $B_\epsilon^n$ from Lemma \ref{LEM_volume_unit_balls}. We then use the Blaschke-Santal\'o inequality (see Lemma \ref{lem:blaschke santalo}) to derive an upper bound on the volume of $B_\pi^n$.

The next proposition will be a consequence of the $3$-fold Chevet inequality, where we apply Lemma \ref{cor chevet} to the space $X_\epsilon^n$ and choose $(x_i)_{i=1}^n, (y_j)_{j=1}^n, (z_k)_{k=1}^n$ to be the standard basis vectors.

\begin{prop}\label{prop chevet}
Let $n\in \N$ and $1\leq p \leq q \leq r \leq \infty$. Let $(\varepsilon_{i,j,k})_{i,j,k=1}^\infty$ be a sequence of independent Bernoulli random variables. Then we have
\begin{align}\label{bound norm inj}
\E\,\bigg\|\sum_{i,j,k=1}^n\e_{i,j,k}~e_{i}\otimes e_j \otimes e_{k}\bigg\|_{\epsilon} \lesssim_{p,q,r} \, n^{\max\big(\frac{1}{p^*},\frac{1}{2}\big) + \max\big(\frac{1}{q^*},\frac{1}{2}\big)+\frac{1}{r^*}-1}.
\end{align}
\end{prop}

\begin{proof}
First, recall that the Rademacher average is smaller than the Gaussian average (see Lemma \ref{lem:gauss dominates rademacher}) and so it is enough to prove inequality \eqref{bound norm inj} with Gaussian random variables. In order to do that, recall the well known fact that, for all $1\leq \alpha < \infty$, and standard Gaussian random variables $g_1,\dots,g_n$,
\begin{align*}
\mathbb E\,\bigg\|\sum_{i=1}^n g_i e_i\bigg\|_{\ell_{\alpha}^n} \asymp_\alpha n^{1/\alpha}.
\end{align*}
Also, it is known that if we consider the standard unit vectors $e_1,\dots,e_n$ in $\ell_\alpha^n$, then we have
\begin{align*}
\big\|(e_i)_{i=1}^n\big\|_{\omega,2} = n ^{\max\big(\frac 1 \alpha, \frac 1 2\big)-\frac 1 2}.
\end{align*}
Then, Lemma \ref{cor chevet} implies
\begin{multline}\label{bound with power}
\E\, \bigg\| \sum_{i,j,k=1}^n\e_{i,j,k}e_{i}\otimes e_j \otimes e_{k}\bigg\|_{\epsilon} \lesssim_{p,q,r}  n^{\max\big(\frac{1}{p^*},\frac{1}{2}\big)+\max\big(\frac{1}{q^*},\frac{1}{2}\big)+\frac{1}{r^*}-1} 
\\
  + n^{\max\big(\frac{1}{p^*},\frac{1}{2}\big)+\max\big(\frac{1}{r^*},\frac{1}{2}\big)+\frac{1}{q^*}-1} + n^{\max\big(\frac{1}{q^*},\frac{1}{2}\big)+\max\big(\frac{1}{r^*},\frac{1}{2}\big)+\frac{1}{p^*}-1} 
\\ \le 3n^{\max\big(\frac{1}{p^*},\frac{1}{2}\big)+\max\big(\frac{1}{q^*},\frac{1}{2}\big)+\frac{1}{r^*}-1},
\end{multline}
where in the last inequality we used the assumption that $p\leq q \leq r$.
\end{proof}

An upper bound on the volume radius $B_{\pi}^n$ is now an immediate consequence of Lemma \ref{LEM_volume_unit_balls}, Proposition \ref{prop chevet}, and the Blaschke-Santal\'o inequality.

\begin{cor}\label{vol upper bound}
Let $n\in\N$ and $1\leq p \leq q \leq r \leq \infty$. Then we have
 \begin{align*}
 \vol_{n^3}\big(B_{\pi}^n\big)^{1/n^3} & \lesssim_{p,q,r}\, n^{-\min\left(\frac{1}{p},\frac1 2\right) -\min\left(\frac{1}{q},\frac1 2\right) -\frac{1}{r}-1}.
 \end{align*}
\end{cor}

\begin{proof}
Applying Proposition \ref{prop chevet} to $X_\epsilon^n$ and using Lemma \ref{LEM_volume_unit_balls}, we get
\begin{align}\label{eq:lower bound injective ball}
\vol_{n^3}\big(B_\epsilon^n\big)^{1/n^3} \gtrsim_{p,q,r} n^{-\max\big(\frac{1}{p^*},\frac{1}{2}\big) - \max\big(\frac{1}{q^*},\frac{1}{2}\big)-\frac{1}{r^*}+1}.
\end{align}
Lemma~\ref{lem:blaschke santalo} implies that
\begin{align}\label{eq:upper bound volume product}
\vol_{n^3}\left(B_{\pi}^n\right)^{1/n^3} \cdot \vol_{n^3}\big(B_\epsilon^n\big)^{1/n^3} \le \Big(\vol_{n^k}\big(B_{2}^{n^3}\big)^2\Big)^{1/n^3} \asymp n^{-3}.
\end{align}
Thus, combining \eqref{eq:lower bound injective ball} and \eqref{eq:upper bound volume product}, we obtain
\begin{align*}
\vol_{n^3}\big(B_{\pi}^n\big)^{1/n^3} \lesssim_{p,q,r}  n^{\max\big(\frac{1}{p^*},\frac 1 2 \big) + \max\big(\frac{1}{q^*},\frac 1 2 \big)+ \frac{1}{r^*}-4} .
\end{align*}
Since $\max\big(\frac 1{p^*},\frac 1 2\big) -1 = -\min\big(\frac 1 p, \frac 1 2\big)$,
we have
\begin{align*}
\vol_{n^3}\big(B_{\pi}^n\big)^{1/n^3} & \lesssim_{p,q,r}\, n^{-\min\left(\frac{1}{p},\frac1 2\right) -\min\left(\frac{1}{q},\frac1 2\right) -\frac{1}{r}-1},
\end{align*}
and the proof is complete.
\end{proof}

\medskip
%%%%%%%%%%%%%%%%%%%%%%%%%%%%%%%%%%
\section{The operator norm of the identity operator}\label{sec norm}
%%%%%%%%%%%%%%%%%%%%%%%%%%%%%%%%%%

In this section we will present upper and lower bounds for $\id\in \mathcal L(\ell_2^{n^3},X^n_\pi)$ that are sharp for most choices of $p,q,r$. The lower bounds are based on Chevet's inequality, on the special structure of the space of diagonal tensors in $X_\pi^n$ or the choice of particular $3$-fold tensors. To obtain upper bounds, we use the results for $2$-fold projective tensor products and a generalized version of an inequality that was proved by Hardy and Littlewood to study bilinear forms. The main theorem in this section reads as follows.

\begin{thmalpha}\label{thm norm iden}
Let $n\in\N$ and $1 \le p \le q \le r \le \infty$. Then we have
\begin{align*}
\|\id:\ell_2^{n^3}\to \X\| \asymp_{p,q,r}
\begin{cases}
n^{\frac 1 2 + \frac 1 r}\,, & r \le 2,
\\
n^{\max\left(\frac 1 q + \frac 1 r, \frac 1 2 \right)}\,, & p \le 2 \le q,\, \frac 1 p + \frac 1 q + \frac 1 r \ge 1,
\\
n^{\frac 1 p + \frac 1 q + \frac 1 r - \frac 1 2}\,, & p \le 2 \le q,\, \frac 1 p + \frac 1 q + \frac 1 r \le 1,
\\
n^{\max\left(\frac 1 p + \frac 1 q + \frac 1 r, \frac 1 2\right) - \frac 1 2 }\,, & 2 \le p.
\end{cases}
\end{align*}
In the case $q \le 2 \le r$, we obtain the following bounds,
\begin{align*}
n^{\frac 1 2 + \frac 1 r} \lesssim_{p,q,r} \|\id:\ell_2^{n^3}\to \X\| \lesssim n^{\min\left(\frac 1 q + \frac 1 r, 1\right)}.
\end{align*}
\end{thmalpha}

\medskip

The lower bound is proved in Section~\ref{sec lower bound norm} and the upper bound in Section~\ref{sec upper bound norm}. Also, in Section~\ref{sec special cases} we show how the lower bound can be improved in some special cases when $q\leq 2 \leq r$ in the previous theorem. In what follows, we will use the shorthand notation $\norm{\id}_{\ell_2^{n^3}\to \X}$ for $\|\id:\ell_2^{n^3}\to \X\|$.

%%%%%%%%%%%%%%%%%%%%%%%%%
\subsection{Lower bounds on the operator norm}\label{sec lower bound norm}
%%%%%%%%%%%%%%%%%%%%%%%%%

The following lower bound can be obtained by considering the space of diagonal tensors in $X_{\pi}^n$ and by using the $3$-fold version of Chevet's inequality (see Proposition \ref{prop chevet}).

\begin{lem} \label{lem lower bound norm}
Let $n\in\N$ and $1 \le p, q, r \le \infty$. Then we have
\begin{align}\label{lower bound idem}
\norm{\id}_{\ell_2^{n^3}\to \X} \gtrsim_{p,q,r} \max\left(n^{\max\left(\min\left(1,\frac 1 {p} + \frac 1 q + \frac 1 r\right) -\frac 1 2,0\right)}, n^{\min\left(\frac{1}{p}, \frac 1 2\right) + \min\left(\frac{1}{q}, \frac 1 2\right) + \frac 1 {r} -\frac 1 2}\right).
\end{align}
\end{lem}
\begin{proof}
By Theorem 1.3 in~\cite{AF96}, it is known that the space of diagonal tensors in $X_{\pi}^n$ is isometric to $\ell_s^{n^3}$, where $s$ is given by
\begin{align*}
s := \frac 1 {\min\left(\frac 1 p + \frac 1 q + \frac 1 r, 1\right)}.
\end{align*}
Hence, we have
\begin{align*}
\norm{\id}_{\ell_2^{n^3}\to \X} \ge \frac{\left\|\sum_{i,j,k=1}^ne_i\otimes e_i \otimes e_i\right\|_{\pi}}{\left\|\sum_{i,j,k=1}^ne_i\otimes e_i \otimes e_i\right\|_{\ell_2^{n^3}}} = \frac{\left\|\sum_{i,j,k=1}^ne_i\otimes e_i \otimes e_i\right\|_{\ell_s^{n^3}}}{\left\|\sum_{i,j,k=1}^ne_i\otimes e_i \otimes e_i\right\|_{\ell_2^{n^3}}} = n^{\frac 1 s -\frac 1 2}.
\end{align*}
Since $\norm{\id}_{\ell_2^{n^k}\to \X} \ge 1$, we have in fact
\begin{align}\label{first lower bound}
\norm{\id}_{\ell_2^{n^k}\to \X} \ge n^{\max\left(\min\left(1,\frac 1 {p} + \frac 1 q + \frac 1 r\right) -\frac 1 2,0\right)}.
\end{align}
Next, notice that by Proposition \ref{prop chevet}, there exists a choice of signs $(\e_{i,j,k})_{i,j,k=1}^n$ such that
\begin{align*}
\left\|\sum_{i,j,k=1}^n~\e_{i,j,k}e_1\otimes e_2 \otimes e_3\right\|_{\epsilon} \lesssim_{p,q,r} n^{\max\left(\frac 1 {p^*},\frac 1 2 \right) + \max\left(\frac 1 {q*},\frac 1 2 \right) + \frac 1 {r^*}- 1 }.
\end{align*}
On the other hand,
\begin{equation*}%\label{eq:2_norm_sign_unit_tensor}
\left\|\sum_{i,j,k=1}^n\e_{i,j,k}~e_{1}\otimes e_2 \otimes e_3\right\|_{\ell_2^{n^3}} = n^{\frac 3 2}.
\end{equation*}
Thus,
\begin{align}\label{second lower bound}
\nonumber \big\|\id\big\|_{\ell_2^{n^k} \to \X} & = \big\|\id\big\|_{(\X)^* \to \ell_2^{n^k}} \ge \frac{\left\|\sum_{i,j,k=1}^n\e_{i,j,k}~e_{1}\otimes e_2 \otimes e_3\right\|_{\ell_2^{n^3}}}{\left\|\sum_{i,j,k=1}^n\e_{i,j,k}~e_{1}\otimes e_2 \otimes e_3\right\|_{\epsilon}}
\\
& \gtrsim_{p,q,r} n^{\frac 5 2 - \max\left(\frac 1 {p^*},\frac 1 2 \right) - \max\left(\frac 1 {q*},\frac 1 2 \right) - \frac 1 {r^*}} \stackrel{(*)}{=} n^{\min\left(\frac 1 {p},\frac 1 2 \right) + \min\left(\frac 1 {q},\frac 1 2 \right) + \frac 1 {r} - \frac 1 2},
\end{align}
where in ($*$) we used the fact that for any $p \ge 1$, $1-\max\big(\frac 1{p^*}, \frac 1 2\big) = \min\big(\frac 1 p, \frac 1 2\big)$. Combining~\eqref{first lower bound} and~\eqref{second lower bound}, the proof is complete.
\end{proof}

%%%%%%%%%%%%%%%%%%%%%%%%%
\subsection{Upper bounds on the operator norm}\label{sec upper bound norm}
%%%%%%%%%%%%%%%%%%%%%%%%%

For tensor products of two spaces, the norm of the identity was estimated in~\cite{Sch82}.

\begin{prop}\label{bound k is 2}
Let $n\in\N$ and $1 \le p \le q \le \infty$. Then
\begin{align*}
\|\id\|_{\ell_2^{n^2} \to \ell_p^n\otimes_\pi\ell_q^n} \asymp_{p,q}
\begin{cases}
 n^{\frac 1 {q}} & q \leq 2,
 \\
 n^{\min\left(\frac 1 {p} + \frac{1}{q},1\right)-\frac{1}{2}} & p \leq 2 \leq q,
 \\
 n^{\max\left(\frac{1}{p}+\frac{1}{q}, \frac 1 2 \right) -\frac{1}{2}} & 2\le p.
\end{cases}
\end{align*}
\end{prop}

To obtain an upper bound for $3$-fold tensor products of $\ell_p$-spaces, we use the following recursive formula.

\begin{prop}\label{prop recursive}
Let $n\in\N$ and $1 \le p,q,r \le \infty$. Then
\begin{align*}
\|\id\|_{\ell_2^{n^3} \to \X} \le n^{\min\left(\frac 1 {r}, \frac 1 2\right)}\cdot \|\id\|_{\ell_2^{n^2} \to \ell_p^n \otimes_{\pi} \ell_q^n}.
\end{align*}
\end{prop}

\begin{proof}
First, note that
\begin{align*}
\|\id\|_{\ell_2^{n^{2}} \to \ell_{p}^n\otimes_{\pi}\ell_{q}^n} = \|\id\|_{(\ell_{p}^n\otimes_{\pi}\ell_{q}^n)^* \to \ell_2^{n^{2}}}.
\end{align*}
Let $A\in (\X)^*$. By the definition of injective tensor product norm, we have
\begin{align}\label{lower bound product}
\|\id\|_{(\ell_{p}^n\otimes_{\pi}\ell_{q}^n)^*\to \ell_2^{n^{2}}}\cdot \|A\|_{(\X)^*} & = \|\id\|_{(\ell_{p}^n\otimes_{\pi}\ell_{q}^n)^* \to \ell_2^{n^{2}}}\cdot \|A\|_{\ell_{r}^n\to (\ell_{p}^n\otimes_{\pi}\ell_{q}^n)^*} \ge \|A\|_{\ell_{r}^n \to \ell_2^{n^{2}}}.
\end{align}
Also, for $x = (x_1,\dots,x_n)\in \ell_{r}^n$, we have
\begin{align*}
Ax = \sum_{i,j=1}^n \left[\sum_{k=1}^nA_{i,j,k}x_k\right] e_{i}\otimes e_j \in \R^{n^{2}}.
\end{align*}
Thus, we have
\begin{align*}
\|A\|_{\ell_{r}^n \to \ell_2^{n^{2}}}^2 = \sup_{\|x\|_{r}=1}\|Ax\|_{\ell_2^{n^{2}}}^2 =  \sup_{\|x\|_{\ell_r^n}=1}\sum_{i,j=1}^n\left|\sum_{k=1}^nA_{i,j,k}x_k\right|^2
\end{align*}
Choosing $x= n^{-1/r}\cdot \e$, $\e\in \{-1,1\}^n$, we get
\begin{align}\label{lower bound ell 2}
\|A\|_{\ell_{r}^n \to \ell_2^{n^{2}}}^2 \ge n^{-\frac{2}{r}}\sum_{i,j=1}^n\mathbb E\left|\sum_{k=1}^nA_{i,j,k}\e_k\right|^2  = n^{-\frac 2{r}}\sum_{i,j,k=1}^n\left|A_{i,j,k}\right|^2 = \left(n^{-1/r}\|A\|_{\ell_2^{n^3}}\right)^2.
\end{align}
Combining~\eqref{lower bound product} and~\eqref{lower bound ell 2}, we get
\begin{align*}
n^{\frac{1}{r}}\,\|\id\|_{(\ell_{p}^n\otimes_{\pi}\ell_{q}^n)^* \to \ell_2^{n^{2}}}\, \|A\|_{\epsilon} \ge \|A\|_{\ell_2^{n^3}}.
\end{align*}
This gives
\begin{align}\label{bound with pk}
\|\id\|_{\ell_2^{n^3} \to \X}  \le n^{\frac{1}{r}}\, \|\id\|_{\ell_2^{n^2} \to \ell_p^n \otimes_{\pi} \ell_q^n}.
\end{align}
Also, since for all $x\in \R^n$, $\|x\|_{\ell_{\infty}^n} \le \|x\|_{\ell_2^n} \le \sqrt{n}\|x\|_{\ell_{\infty}^n}$, we also have
\begin{align*}
\|\id\|_{ (\ell_{p}^n\otimes_{\pi}\ell_{q}^n)^* \to \ell_2^{n^{2}}}^2\,\|A\|_{\epsilon}^2 & \ge \sup_{\|x\|_{\ell_r^n}=1}\sum_{i,j=1}^n\left|\sum_{k=1}^nA_{i,j,k}x_k\right|^2 \ge \max_{1\le k \le n}\sum_{i,j=1}^n\left|A_{i,j,k}\right|^2 \ge \frac 1 n \|A\|_{\ell_2^{n^3}},
\end{align*}
which implies
\begin{align}\label{bound with sqrt}
\|\id\|_{\ell_2^{n^3} \to \X}  \le \sqrt{n}\, \|\id\|_{\ell_2^{n^2} \to \ell_p^n \otimes_{\pi} \ell_q^n}.
\end{align}
Combining \eqref{bound with pk} and \eqref{bound with sqrt}, the result follows.
\end{proof}

Another useful tool is the following corollary, which follows immediately from Theorem~\ref{thm multilinear} in Section~\ref{sec prelim}.

\begin{cor}\label{cor multi}
Let $n\in\N$ and assume that $1 \le p,q,r \le \infty$ are such that $\frac 1 p + \frac 1 q + \frac 1 r \le \frac 1 2$. Then
\begin{align*}
\|\id\|_{\ell_2^{n^3}\to \X} \lesssim 1.
\end{align*}
\end{cor}

\begin{proof}
Since we assume that $\frac 1 p + \frac 1 q + \frac 1 r \le \frac 1 2$, $\mu$ as defined in~\eqref{def mu}, we have $\mu \le 2$. Hence, we have
\begin{align*}
\|A\|_{\ell_2^{n^3}} \le \|A\|_{\ell_{\mu}^{n^3}} \lesssim \|A\|_{(\X)^*}.
\end{align*}
Therefore, we have
\begin{align*}
\|\id\|_{\ell_2^{n^3}\to \X} = \|\id\|_{(\X)^* \to  \ell_2^{n^3}} \lesssim 1,
\end{align*}
which completes the proof.
\end{proof}

Using the above upper bounds, we obtain the following.

\begin{lem}\label{lem upper bound norm}
Let $n\in\N$ and $1 \le p \le q \le r \le \infty$. Then we have
\begin{numcases}{\|\id\|_{\ell_2^{n^3}\to \X} \lesssim_{p,q,r}}
n^{\frac 1 2 + \frac 1 r}, & $r \le 2$, \label{r less 2}
\\
n^{\min\left(\frac 1 q + \frac 1 r, 1\right)}, & $q \le 2 \le r$, \label{q less 2 less r}
\\
n^{\max\left(\frac 1 q + \frac 1 r, \frac 1 2 \right)}, & $p \le 2 \le q$, $\frac 1 p + \frac 1 q + \frac 1 r \ge 1$, \label{p less 2 less q sum big}
\\
n^{\frac 1 p + \frac 1 q + \frac 1 r - \frac 1 2}, & $p \le 2 \le q$, $\frac 1 p + \frac 1 q + \frac 1 r \le 1$, \label{p less 2 less q sum small}
\\
n^{\max\left(\frac 1 p + \frac 1 q + \frac 1 r, \frac 1 2\right) - \frac 1 2 }, & $2 \le p$. \label{p bigger 2}
\end{numcases}
\end{lem}

\begin{proof}
The bounds~\eqref{r less 2},~\eqref{q less 2 less r},~\eqref{p less 2 less q sum big} follow from Proposition~\ref{bound k is 2} and Proposition~\ref{prop recursive}. Next, assume that $p,q,r$ are such that $\frac 1 2 \le \frac 1 p + \frac 1 q + \frac 1 r \le 1$. Then there exists $\tilde p \ge p$, $\tilde q \ge q$, $\tilde r \ge r$ such that $\frac 1 {\tilde p} + \frac 1 {\tilde q} + \frac 1 {\tilde r} = \frac 1 2$. Hence, given $A \in \X$, we have
\begin{align*}
\|A\|_{\pi} & \le n^{\frac 1 p + \frac 1 q + \frac 1 r -\frac 1 2}\|A\|_{\ell_{\tilde p}^n \otimes_{\pi} \ell_{\tilde q}^n \otimes_{\pi} \ell_{\tilde r}^n }
\\
& \le n^{\frac 1 p + \frac 1 q + \frac 1 r -\frac 1 2}\|\id\|_{\ell_2^{n^3}\to \ell_{\tilde p}^n \otimes_{\pi} \ell_{\tilde q}^n \otimes_{\pi} \ell_{\tilde r}^n}\|A\|_{\ell_2^{n^3}}\\
&  \lesssim n^{\frac 1 p + \frac 1 q + \frac 1 r -\frac 1 2}\|A\|_{\ell_2^{n^3}},
\end{align*}
where in the last inequality we used Corollary~\ref{cor multi}. Hence, we have
\begin{align}\label{case sum between 1/2 and 1}
\|\id\|_{\ell_2^{n^3}\to \X} \lesssim n^{\frac 1 p + \frac 1 q + \frac 1 r - \frac 1 2}.
\end{align}
Thus, if we assume that $p \le 2 \le q$ and $\frac 1 p + \frac 1 q + \frac 1 r \le 1$, then we must have $\frac 1 2 \le \frac 1 p + \frac 1 q + \frac 1 r \le 1$, in which case~\eqref{case sum between 1/2 and 1} proves~\eqref{p less 2 less q sum small}. Finally, assume that $p \ge 2$. If $\frac 1 2 \le \frac 1 p + \frac 1 q + \frac 1 r \le 1$, then~\eqref{case sum between 1/2 and 1} holds. If $\frac 1 p + \frac 1 q + \frac 1 r \le \frac 1 2$, then by Corollary~\ref{cor multi},
\begin{align}\label{case sum less 1/2}
\|\id\|_{\ell_2^{n^3}\to \X} \lesssim 1.
\end{align}
If we have $\frac 1 p + \frac 1 q + \frac 1 r \ge 1$, then since $p \ge 2$ we must have $\frac 1 q + \frac 1 r \ge \frac 1 2$. Hence, using Proposition~\ref{bound k is 2} and Proposition~\ref{prop recursive}, we have
\begin{align}\label{case sum bigger 1}
\|\id\|_{\ell_2^{n^3}\to \X} \lesssim  n^{\frac 1 p + \max\left(\frac 1 q + \frac 1 r, \frac 1 2\right) -\frac 1 2} = n^{\frac 1 p + \frac 1 q + \frac 1 r - \frac 1 2}.
\end{align}
Combining~\eqref{case sum between 1/2 and 1},~\eqref{case sum less 1/2}, and~\eqref{case sum bigger 1},~\eqref{p bigger 2} follows.
\end{proof}

%%%%%%%%%%%%%%%%%%%%%%%%%%%%%%%
\subsection{Improved lower bounds for some special cases}\label{sec special cases}
%%%%%%%%%%%%%%%%%%%%%%%%%%%%%%%

In this short section we show that the lower bound for the case $p \le q \le 2 \le r$ in Theorem \ref{main thm} can be improved for some particular choices of $p,q$ and $r$.

\begin{lem}\label{lem special cases}
Let $n\in\N$ and $1 \le p \le q \le 2 \le r \le \infty$. Then, we have
\begin{align*}
\| \id \|_{\ell_2^{n^3} \to \X} \ge \max\left(n^{\frac 1 2 + \frac 1 r}, n^{\min\left(\frac 1 q + \frac 1 r, 1\right) + \frac 1 p -1}, n^{\frac 1 q}\right).
\end{align*}
\end{lem}

\begin{proof}
The first bound follows from Lemma~\ref{lem lower bound norm}. In order to prove the second bound, we transfer to a $3$-fold tensor that is constructed from a $2$-fold tensor $B \in \ell_{q}^n \otimes_{\pi} \ell_{r}^n$ and the special vector $x = (1,1,\dots,1) \in \ell_{p}^n$. We have
\begin{align*}
\| \id \|_{\ell_2^{n^3} \to \X} & \ge \sup_{B \in \ell_{q}^n \otimes_{\pi} \ell_{r}^n} \frac{\|x\otimes B\|_{\pi}}{\|x\otimes B\|_{\ell_2^{n^3}}} \\
& =  \sup_{B \in \ell_{q}^n \otimes_{\pi} \ell_{r}^n} \frac{\|B\|_{\ell_{q}^n \otimes_{\pi} \ell_{r}^n}\|x\|_{\ell_r^n}}{\|B\|_{\ell_2^{n^2}}\|x\|_{\ell_2^n}} \\
& = n^{\frac 1 p -\frac 1 2}\|\id\|_{\ell_2^{n^2} \to \ell_{q}^n \otimes_{\pi} \ell_{r}^n} \\
& =  n^{\min\left(\frac 1 q + \frac 1 r, 1\right) + \frac 1 p -1},
\end{align*}
where in the last equality we used Proposition~\ref{bound k is 2}. To obtain the third bound, we again transfer to the $2$-fold case, choosing $B \in \ell_p^n\otimes_{\pi} \ell_q^n$ and $x = (1,0,\dots,0)\in\ell_r^n$. Then we have
\begin{align*}
\big\| \id \|_{\ell_2^{n^3} \to \X} \ge \sup_{B \in \ell_{p}^n \otimes_{\pi} \ell_{q}^n} \frac{\|x\otimes B\|_{\pi}}{\|x\otimes B\|_{\ell_2^{n^3}}} =  \sup_{B \in \ell_{p}^n \otimes_{\pi} \ell_{q}^n} \frac{\|B\|_{\ell_{p}^n \otimes_{\pi} \ell_{q}^n}\|x\|_{\ell_r^n}}{\|B\|_{\ell_2^{n^2}}\|x\|_{\ell_2^n}} = \|\id\|_{\ell_2^{n^2} \to \ell_{p}^n \otimes_{\pi} \ell_{q}^n} =  n^{\frac 1 q},
\end{align*}
where again in the last equality we used Proposition~\ref{bound k is 2}. This completes the proof.
\end{proof}

Combining Lemma~\ref{lem special cases} with Lemma~\ref{lem upper bound norm}, we obtain the following improvements on the bounds in particular cases.

\begin{cor}\label{cor:improvement open case}
(i) If $p=1 \le q \le 2 \le r$, then
\begin{align*}
\|\id\|_{\ell_2^{n^3}\to \X} \asymp_{q,r} n^{\min\left(\frac 1 q + \frac 1 r,1\right)}.
\end{align*}
In particular, in such case we have
\begin{align*}
\vr\left(\ell_p^n \otimes_{\pi} \ell_q^n \otimes_{\pi} \ell_r^n\right) \asymp_{q,r}  n^{\min\left(\frac 1 q - \frac 1 2, \frac 1 2 - \frac 1 r\right)}.
\end{align*}
(ii) If $p=q \le 2$ and $r=\infty$, then
\begin{align*}
\|\id\|_{\ell_2^{n^3}\to \X} \asymp_{p} n^{\frac 1 q}.
\end{align*}
In particular, in such case we have
\begin{align*}
\vr\left(\ell_p^n \otimes_{\pi} \ell_q^n \otimes_{\pi} \ell_r^n\right) \asymp_{p} n^{\frac 1 q - \frac 1 2}.
\end{align*}
(iii) If $1\leq p\leq q = 2 \leq r$, then
\begin{align*}
\|\id\|_{\ell_2^{n^3}\to \X} \asymp_{p,r} n^{\frac 1 2 + \frac 1 r}.
\end{align*}
In particular, in such case we have
\begin{align*}
\vr\left(\ell_p^n \otimes_{\pi} \ell_q^n \otimes_{\pi} \ell_r^n\right) \asymp_{p,r} 1.
\end{align*}
\end{cor}

Corollary~\ref{cor:improvement open case} suggests that it is the lower bound that should be improved. See also Corollary~\ref{cor all p same} below for another such indication.

\medskip
%%%%%%%%%%%%%%%%%%%%%%%%%%%
\section{The case of $k$-fold projective tensor products}\label{sec k fold}
%%%%%%%%%%%%%%%%%%%%%%%%%%%

In this section, we briefly present the generalization of our main result to the case of $k$-fold projective tensor products. Before stating the generalized main result, fix some notation. Given $k\in\N$, $1\le p_1\le p_2\le \dots \le p_k \le \infty$ and $n\in \mathbb N$, let
\begin{align}\label{def k fold}
X_{\p}^n := \ell_{p_1}^n\otimes_{\pi}\dots\otimes_{\pi}\ell_{p_k}^n.
\end{align}
Also, let $j_0$ be the largest $1 \le j \le k$ such that $p_j \le 2$ ($j_0 = 0$ if $p_1>2$). The space $X_{\p}^n$ has enough symmetries, and so formula \eqref{formula vol ell} gives
\begin{align}\label{formula vr 2}
\vr(X_{\p}^n) \asymp n^{k/2}\left(\vol_{n^k}(B_{X_{\p}^n})\right)^{1/n^k} \big\|\id\big\|_{\ell_2^n\to X_{\p}^n}.
\end{align}
In what follows, $\gtrsim_{\p}$ and $\lesssim_{\p}$ mean inequalities with an implied constant that depends only on $p_1,\dots,p_k$. All the tools that were used above can be generalized to the case of $k$-fold tensor products by straightforward induction.

Regarding the volume of $B_{X_{\p}^n}$, note that in the $k$-fold version of Chevet's inequality (see Lemma~\ref{cor chevet k-fold} below for a complete statement), there is an additional $k$ factor since the upper bound is now comprised of $k$ terms, and also note that the lower bound of the volume of $B_{X_{\p}}^n$ now contains a factor of $2^{-\frac{k-j_0}{2}}$, since for each $j$ with $p_j \ge 2$, the use of Khinchine's inequality incurs a factor of $1/\sqrt{2}$. As a result, we have in the general case,
\begin{align*}
2^{-\frac{k-j_0}{2}}\, n^{-\sum_{j=1}^{k-1}\min\left(\frac 1{p_j}, \frac 1 2\right) -\frac 1 {p_k}-1} \lesssim_{\p} \vol_{n^k}(B_{X_{\p}^n}) \lesssim_{\p} k\, n^{-\sum_{j=1}^{k-1}\min\left(\frac 1{p_j}, \frac 1 2\right) -\frac 1 {p_k}-1}.
\end{align*}

Regarding the norm of the identity, an analogue of Theorem~\ref{thm norm iden} follows by simply using induction. The only exception is Theorem~\ref{thm multilinear}, which in the general case still gives $\|\id\|_{\ell_2^{n^k} \to X_{\p}^n} \lesssim 1$ whenever $\sum_{j=1}^k \frac{1}{p_j} \le \frac 1 2$.

Combining those tools, Theorem~\ref{main thm} can be generalizes in the following way:

\begin{thmalpha}\label{thm: k fold}
        Let $k\ge 3$. Let $1\le p_1\le p_2\le \dots \le p_k \le \infty$ and $n\in \mathbb N$. Then the following estimates hold:
        \begin{enumerate}
                \setlength{\itemsep 5pt}
                \item If $p_k \le 2$, then
                \begin{align*}
                1 \lesssim_{\p} \vr(\X) \lesssim_{\p} k.
                \end{align*}
                \item If $\sum_{j=1}^k\frac 1 {p_j} \ge 1$, $p_{k-1}\le 2 < p_k$, then
                \begin{align*}
                1  \lesssim_{\p} \vr(\X) \lesssim_{\p} k\, n^{\min\left(\frac 1 {p_{k-1}}+\frac 1 {p_k},1\right)-\frac 1{p_{k}}-\frac 1 2}.
                \end{align*}
                \item If $\sum_{j=1}^k\frac 1 {p_j} \ge 1$, $p_{k-1}>2$, then
                \begin{align*}
                1  \lesssim_{\p} \vr(\X) \lesssim_{\p} k\, n^{\max\left(\frac 1 {p_{k-1}}+\frac 1 {p_k},\frac 1 2\right)-\frac 1{p_{k-1}}-\frac 1 {p_k}}
                \end{align*}
                \item If $\sum_{j=1}^k\frac 1 {p_j} \ge 1$,  and If $\sum_{j=2}^k\frac 1 {p_j} \le \frac 1 2$, then
                \begin{align*}
                2^{-\frac k 2} n^{\frac 1 2  -\sum_{j=2}^k\frac 1 {p_j}} \lesssim_{\p} \vr(\X) \lesssim_{\p} k\, n^{\frac 1 2  -\sum_{j=2}^k\frac 1 {p_j}}
                \end{align*}
                \item If $\frac 1 2 \le \sum_{j=1}^k\frac 1 {p_j} \le 1$, then
                \begin{align*}
                2^{-\frac k 2} n^{\max\left(\frac 1 {p_1}-\frac 1 2,0\right)}  \lesssim_{\p} \vr(\X) \lesssim_{\p} k \, n^{\max\left(\frac 1 {p_1}-\frac 1 2,0\right)}.
                \end{align*}
                \item If $\sum_{j=1}^k\frac 1 {p_j} \le \frac 1 2$, then
                \begin{align*}
                2^{-\frac k 2}  n^{\frac 1 2 -\sum_{j=1}^k\frac 1 {p_j}} \lesssim_{\p} \vr(\X) \lesssim_{\p} k\, n^{\frac 1 2-\sum_{j=1}^k\frac 1 {p_j}}.
                \end{align*}
        \end{enumerate}
\end{thmalpha}

Considering copies of the same $\ell_p^n$ space, it was shown in \cite{DP09}, that if one considers the space
\begin{align*}
\otimes_{\pi}^k\ell_p^n := \underbrace{\ell_p^n\otimes_{\pi}\ell_p^n\otimes_{\pi}\dots\otimes_{\pi}\ell_p^n}_\text{$k$ times},
\end{align*}
then the following holds true:
\begin{align}\label{thm all p same}
\vr\left(\otimes_{\pi}^k\ell_p^n\right) \asymp_p \begin{cases} 1 & p \le 2k, \\ n^{\frac 1 2 - \frac k p} & p \ge 2k. \end{cases}
\end{align}

Using Theorem \ref{thm: k fold}, we can get a result in the spirit of \eqref{thm all p same}, with worse dependence on $k$, and in some case, with worse dependence on $n$ as well:

\begin{cor}\label{cor all p same}
        The following holds:
        \begin{enumerate}
                \item If $p\le 4$ or $k\le p \le 2k$ then $1 \lesssim_{p} \vr \left(\otimes_{\pi}^k\ell_p^n\right) \lesssim_{p} k$.
                \item If $4 \le p \le k $ then $1 \lesssim_{p} \vr \left(\otimes_{\pi}^k\ell_p^n\right) \lesssim_{p} k\, n^{\frac 1 2 -\frac 2 p}$.
                \item If $p \ge 2k$ then $2^{-\frac k 2}\, n^{\frac 1 2 -\frac k p} \lesssim_{p} \vr \left(\otimes_{\pi}^k\ell_p^n\right) \lesssim_{p} k\, n^{\frac 1 2-\frac k {p}}$.
        \end{enumerate}
\end{cor}

Corollary~\ref{cor all p same} combined with~\eqref{thm all p same} suggest that it is the lower bound that should be improved in Theorem~\ref{thm: k fold}. Finally, for the sake of completeness, let us state the $k$-fold version of Chevet's inequality that plays a major role in the proof of Theorem \ref{thm: k fold}:

\begin{lem}\label{cor chevet k-fold}
        Let $k\in\N$ and $\{x_{i_j}\}_{i_j=1}^{n_j} \subseteq X_j$, $1 \le j \le k$ be sequences in the Banach spaces $(X_1,\|\cdot\|_{X_1}),\dots,(X_k, \|\cdot\|_{X_k})$, respectively. Then,
        \begin{align*}
        \left\|\sum_{\substack{1\le i_m \le n \\ 1 \le m \le k}}g_{i_1,\dots,i_k}x_{i_1}\otimes\dots\otimes x_{i_k}\right\|_{X_1\otimes_{\epsilon}\dots\otimes_{\epsilon}X_k} \le \sum_{j=1}^k \left[\prod_{j' \neq j} \big\|\{x_{i_j}\}_{i_j}\big\|_{\omega,2} \right]\mathbb E\bigg\|\sum_{i_j=1}^{n}g_{i_j}x_{i_j}\bigg\|_{X_j}.
        \end{align*}
\end{lem}

\medskip
%%%%%%%%%%%%%%%%%%%%%%%%%%%%%%
\section{Remarks on the cotype of projective tensor products}\label{sec applications}
%%%%%%%%%%%%%%%%%%%%%%%%%%%%%%

As already mentioned in the introduction, from our main result and its $k$-fold generalization, we obtain information on the cotype of the $3$-fold and $k$-fold projective tensor products.
  
Consider now the infinite dimensional space 
\[X := \ell_p\otimes_{\pi}\ell_q\otimes_{\pi}\ell_r,\] 
where $1\le p \le q \le r \le \infty$. Since the space $\X = \ell_p^n \otimes_{\pi}\ell_q \otimes_{\pi}\ell_r^n$ is a finite dimensional subspace of $X$, it follows by Theorem~\ref{thm BM} that if
\begin{align*}
\lim_{n\to \infty}\vr\left(\ell_p^n \otimes_{\pi}\ell_q \otimes_{\pi}\ell_r^n\right) = \infty,
\end{align*}
then $X$ does not have cotype 2, that is $C_2(X) = \infty$. In order to obtain a result about cotype $\alpha>2$, recall the following result from~\cite{TJ79}: if $(X,\|\cdot\|_X)$ is $n$-dimensional, then there exist $x_1,\dots,x_n \in X$ such that
\begin{align*}
\Big(\sum_{i=1}^n\|x_i\|_X^2\Big)^{1/2} \gtrsim C_2(X)\mathbb E\,\Big\|\sum_{i=1}^n\e_ix_i\Big\|_X.
\end{align*}
In particular, in the space $\X = \ell_p^n \otimes_{\pi}\ell_q^n\otimes_{\pi}\ell_r^n$, there exist $x_1,\dots,x_{n^3} \in \X$ such that
\begin{multline*}
C_2(\X)\,\mathbb E\,\Big\|\sum_{i=1}^{n^3}\e_ix_i\Big\|_{\pi} \lesssim \Big(\sum_{i=1}^{n^3}\|x_i\|_\pi^2\Big)^{1/2}
\\ \stackrel{(*)}{\le} n^{3\left(\frac 1 2 - \frac 1 \alpha\right)}\Big(\sum_{i=1}^{n^3}\|x_i\|_\pi^\alpha\Big)^{1/\alpha} \le n^{3\left(\frac 1 2 - \frac 1 \alpha\right)}C_{\alpha}(\X)\,\mathbb E\,\Big\|\sum_{i=1}^{n^3}\e_ix_i\Big\|_{\pi},
\end{multline*}
where in ($*$) we used H\"older's inequality. Altogether, we have $C_\alpha(\X) \gtrsim n^{3\left(\frac 1 \alpha - \frac 1 2\right)}C_2(\X)$. This fact, together with Theorem \ref{main thm}, gives the following result.

\begin{cor}\label{cor cotype}
        Let $X = \ell_p\otimes_{\pi}\ell_q\otimes_{\pi}\ell_r$ with $1\le p \le q \le r \le \infty$. Then, $C_{\alpha}(X) = \infty$ in the following cases:
        \renewcommand\labelitemi{\tiny$\bullet$}
        \begin{itemize}
                \item $p \le 2$, $\frac 1 q + \frac 1 r < \frac 1 2$, and $\frac 1 p + \frac 1 q + \frac 1 r \ge 1$, and for $\alpha < \frac{3}{1+\frac 1 q + \frac 1 r}$\,;
                \item $p < 2$, $\frac 1 p + \frac 1 q + \frac 1 r \le 1$, and for $\alpha < \frac{3}{2 - \frac 1 p}$\,;
                \item $\frac 1 p + \frac 1 q + \frac 1 r < \frac 1 2$, and for $\alpha < \frac 3 {1+ \frac 1 p + \frac 1 q + \frac 1 r}$\,.
        \end{itemize}
\end{cor}

Finally, we remark that in the case of $k$-fold tensor products, if we now let
\begin{align*}
X := \ell_{p_1}\otimes_{\pi} \dots \otimes_{\pi} \ell_{p_k},
\end{align*}
where $1 \le p_1 \le \dots \le p_k \le \infty$. If $X_\p^n$ is defined as in~\eqref{def k fold}, then a similar argument as above implies that $C_\alpha(X_{\p}^n) \gtrsim n^{k\left(\frac 1 {\alpha} - \frac 1 2\right)}C_2(X_{\p}^n)$. Thus, using Theorem~\ref{thm: k fold}, a similar result to Corollary~\ref{cor cotype} can be obtained. The result reads as follows.

\begin{cor}
        Let $X = \ell_{p_1}\otimes_{\pi}\dots \otimes_{\pi}\ell_{p_k}$ with $1 \le p_1 \le \dots \le p_k \le \infty$. Then, $C_{\alpha}(X) = \infty$ in the following cases:
        \renewcommand\labelitemi{\tiny$\bullet$}
        \begin{itemize}
                \item $\sum_{j=1}^k \frac 1 {p_j} \ge 1$ and $\sum_{j=2}^k \frac 1 {p_j} < \frac 1 2$ and for $\alpha < \frac{k}{\frac{k-1}2 + \sum_{j=2}^k\frac 1 {p_j}}$\,;
                \item $\sum_{j=1}^k\frac 1 {p_j} \le 1$ and $p_1<2$ and for $\alpha < \frac{k}{\frac{k+1}{2}-\frac 1{p_1}}$\,;
                \item $\sum_{j=1}^k \frac 1 {p_j} <\frac 1 2$ and for $\alpha < \frac{k}{\frac{k-1}{2}+\sum_{j=1}^n\frac 1 {p_j}}$\,.
        \end{itemize}
\end{cor}


\begin{thebibliography}{99}

\bib{ASW11}{article}{
        author={Aubrun, G.},
        author={Szarek, S.},
        author={Werner, E.},
        title={Hastings's Additivity Counterexample via Dvoretzky's Theorem},
        journal={Commun. Math. Phys.},
        year={2011},
        volume={305},
        number={1},
        pages={85--97},
        issn={1432-0916},
        %doi="10.1007/s00220-010-1172-y",
        %url="http://dx.doi.org/10.1007/s00220-010-1172-y"
}


\bib{ASW10}{article}{
        author={Aubrun, G.},
        author={Szarek, S.},
        author={Werner, E.},
        title={Nonadditivity of R\'enyi entropy and Dvoretzky's theorem}, 
 journal={J. Math. Phys.},
        year={2010},
        volume={51},
        number={2},
        pages={},
        issn={},
        %doi="10.1007/s00220-010-1172-y",
        %url="http://dx.doi.org/10.1007/s00220-010-1172-y"
}




\bib{AF96}{article}{
   author={Arias, A.},
   author={Farmer, J. D.},
   title={On the structure of tensor products of $l_p$-spaces},
   journal={Pacific J. Math.},
   volume={175},
   date={1996},
   number={1},
   pages={13--37},
   issn={0030-8730},
   %review={\MR{1419470}},
}

\bib{BM87}{article}{
   author={Bourgain, J.},
   author={Milman, V. D.},
   title={New volume ratio properties for convex symmetric bodies in ${\bf
   R}^n$},
   journal={Invent. Math.},
   volume={88},
   date={1987},
   number={2},
   pages={319--340},
   issn={0020-9910},
   %review={\MR{880954}},
   %doi={10.1007/BF01388911},
}

\bib{BNR12}{article}{
   author={Bri{\"e}t, J.},
   author={Naor, A.},
   author={Regev, O.},
   title={Locally decodable codes and the failure of cotype for projective
   tensor products},
   journal={Electron. Res. Announc. Math. Sci.},
   volume={19},
   date={2012},
   pages={120--130},
   issn={1935-9179},
   %review={\MR{2999057}},
   %doi={10.3934/era.2012.19.120},
}

\bib{Che78}{article}{
   author={Chevet, S.},
   title={S\'eries de variables al\'eatoires gaussiennes \`a valeurs dans
   $E\hat \otimes _{\varepsilon }F$. Application aux produits d'espaces de
   Wiener abstraits},
   language={French},
   conference={
      title={S\'eminaire sur la G\'eom\'etrie des Espaces de Banach
      (1977--1978)},
   },
   book={
      publisher={\'Ecole Polytech., Palaiseau},
   },
   date={1978},
   pages={Exp. No. 19, 15},
   %review={\MR{520217}},
}

\bib{DM06}{article}{
   author={Defant, A.},
   author={Michels, C.},
   title={Norms of tensor product identities},
   journal={Note Mat.},
   volume={25},
   date={2005/06},
   number={1},
   pages={129--166},
   issn={1123-2536},
   %review={\MR{2220456}},
}

\bib{DP09}{article}{
   author={Defant, A.},
   author={Prengel, C.},
   title={Volume estimates in spaces of homogeneous polynomials},
   journal={Math. Z.},
   volume={261},
   date={2009},
   number={4},
   pages={909--932},
   issn={0025-5874},
   %review={\MR{2480764}},
   %doi={10.1007/s00209-008-0358-x},
}

\bib{DFS08}{book}{
   author={Diestel, J.},
   author={Fourie, J. H.},
   author={Swart, J.},
   title={The metric theory of tensor products},
   note={Grothendieck's r\'esum\'e revisited},
   publisher={American Mathematical Society, Providence, RI},
   date={2008},
   pages={x+278},
   isbn={978-0-8218-4440-3},
   %review={\MR{2428264}},
   %doi={10.1090/mbk/052},
}

\bib{E09}{article}{
   author={Efremenko, K.},
   title={3-query locally decodable codes of subexponential length},
   conference={
      title={STOC'09---Proceedings of the 2009 ACM International Symposium
      on Theory of Computing},
   },
   book={
      publisher={ACM, New York},
   },
   date={2009},
   pages={39--44},
   %review={\MR{2780048}},
}

\bib{GL74}{article}{
   author={Gordon, Y.},
   author={Lewis, D.R.},
   title={Absolutely summing operators and local unconditional structure},
   journal={Acta Mathematica},
   volume={133},
   date={1974},
   pages={24--48},
}


\bib{Gro53}{article}{
   author={Grothendieck, A.},
   title={R\'esum\'e de la th\'eorie m\'etrique des produits tensoriels
   topologiques},
   language={French},
   journal={Bol. Soc. Mat. S\~ao Paulo},
   volume={8},
   date={1953},
   pages={1--79},
   %review={\MR{0094682}},
}

\bib{HJ74}{article}{
   author={Hoffmann-J{\o}rgensen, J.},
   title={Sums of independent Banach space valued random variables},
   journal={Studia Math.},
   volume={52},
   date={1974},
   pages={159--186},
   issn={0039-3223},
   %review={\MR{0356155}},
}

\bib{HL34}{article}{
        author={Hardy, G. H.},
        author={Littlewood, J. E.},
        title={Bilinear forms bounded in space $[p,q]$},
        journal={Q. J. Math.},
        date={1934},
        pages={241--254},
        volume={5},
        number={1},
}

\bib{Kas77}{article}{
   author={Ka{\v{s}}in, B. S.},
   title={The widths of certain finite-dimensional sets and classes of
   smooth functions},
   language={Russian},
   journal={Izv. Akad. Nauk SSSR Ser. Mat.},
   volume={41},
   date={1977},
   number={2},
   pages={334--351, 478},
   issn={0373-2436},
   %review={\MR{0481792}},
}

\bib{LP68}{article}{
   author={Lindenstrauss, J.},
   author={Pe{\l}czy{\'n}ski, A.},
   title={Absolutely summing operators in $L_{p}$-spaces and their
   applications},
   journal={Studia Math.},
   volume={29},
   date={1968},
   pages={275--326},
   %issn={0039-3223},
   %review={\MR{0231188}},
}

\bib{MP76}{article}{
   author={Maurey, B.},
   author={Pisier, G.},
   title={S\'eries de variables al\'eatoires vectorielles ind\'ependantes et
   propri\'et\'es g\'eom\'etriques des espaces de Banach},
   language={French},
   journal={Studia Math.},
   volume={58},
   date={1976},
   number={1},
   pages={45--90},
   issn={0039-3223},
   %review={\MR{0443015}},
}

\bib{Pis86}{article}{
   author={Pisier, G.},
   title={Probabilistic methods in the geometry of Banach spaces},
   conference={
      title={Probability and analysis},
      address={Varenna},
      date={1985},
   },
   book={
      series={Lecture Notes in Math.},
      volume={1206},
      publisher={Springer, Berlin},
   },
   date={1986},
   pages={167--241},
   %review={\MR{864714}},
   %doi={10.1007/BFb0076302},
}

\bib{Pis89}{book}{
   author={Pisier, G.},
   title={The volume of convex bodies and Banach space geometry},
   series={Cambridge Tracts in Mathematics},
   volume={94},
   publisher={Cambridge University Press, Cambridge},
   date={1989},
   pages={xvi+250},
   isbn={0-521-36465-5},
   isbn={0-521-66635-X},
   %review={\MR{1036275}},
   %doi={10.1017/CBO9780511662454},
}

\bib{P90}{article}{
        title = {Random series of trace class operators.},
        journal = {In Proceedings Cuarto CLAPEM Mexico 1990. Contribuciones en probabilidad y estadistica matematica},
        volume = {},
        number = {},
        pages = {29--42},
        year = {1990},
        %note = {Available at http://arxiv.org/abs/1103.2090.},
        issn = {},
        author = {Pisier, G.},
}

\bib{P92}{article}{
        title = {Factorization of operator valued analytic functions},
        journal = {Adv. Math.},
        volume = {93},
        number = {1},
        pages = {61--125},
        year = {1992},
        issn = {0001-8708},
        author = {Pisier, G.},
}

\bib{Pis12}{article}{
   author={Pisier, G.},
   title={Grothendieck's theorem, past and present},
   journal={Bull. Amer. Math. Soc. (N.S.)},
   volume={49},
   date={2012},
   number={2},
   pages={237--323},
   issn={0273-0979},
   %review={\MR{2888168}},
   %doi={10.1090/S0273-0979-2011-01348-9},
}

\bib{PP81}{article}{
   author={Praciano-Pereira, T.},
   title={On bounded multilinear forms on a class of $l^{p}$ spaces},
   journal={J. Math. Anal. Appl.},
   volume={81},
   date={1981},
   number={2},
   pages={561--568},
   issn={0022-247X},
   %review={\MR{622837}},
   %doi={10.1016/0022-247X(81)90082-2},
}

\bib{Rya02}{book}{
   author={Ryan, R. A.},
   title={Introduction to tensor products of Banach spaces},
   series={Springer Monographs in Mathematics},
   publisher={Springer-Verlag London, Ltd., London},
   date={2002},
   pages={xiv+225},
   isbn={1-85233-437-1},
   %review={\MR{1888309}},
   %doi={10.1007/978-1-4471-3903-4},
}

\bib{Schn14}{book}{
author={Schneider, R.},
title={Convex bodies: the Brunn-Minkowski theory}, 
  series={Encyclopedia of Mathematics and its Applications},
   volume={151},
   publisher={Cambridge University Press},
   date={2014}
}




\bib{Sch82}{article}{
   author={Sch{\"u}tt, C.},
   title={On the volume of unit balls in Banach spaces},
   journal={Compositio Math.},
   volume={47},
   date={1982},
   number={3},
   pages={393--407},
   issn={0010-437X},
   %review={\MR{681616}},
}

\bib{Sza78}{article}{
   author={Szarek, S.},
   title={On Kashin's almost Euclidean orthogonal decomposition of
   $l^{1}_{n}$},
   language={English, with Russian summary},
   journal={Bull. Acad. Polon. Sci. S\'er. Sci. Math. Astronom. Phys.},
   volume={26},
   date={1978},
   number={8},
   pages={691--694},
   issn={0001-4117},
   %review={\MR{518996}},
}

\bib{SWZ11}{article}{
        author={Szarek, S.},
        author={Werner, E.},
        author={Zyczkowski, K.}
        title={How often is a random quantum state $k$--entangled?}
 journal={J. Phys. A},
   volume={40},
   date={2011},
   number={44},
   pages={},
   issn={},
   %review={\MR{571056}},
}



\bib{STJ80}{article}{
   author={Szarek, S.},
   author={Tomczak-Jaegermann, N.},
   title={On nearly Euclidean decomposition for some classes of Banach
   spaces},
   journal={Compositio Math.},
   volume={40},
   date={1980},
   number={3},
   pages={367--385},
   issn={0010-437X},
   %review={\MR{571056}},
}

\bib{T74}{article}{
        author = {Tomczak-Jaegermann, N.},
        journal = {Studia Math.},
        number = {2},
        pages = {163--182},
        title = {The moduli of smoothness and convexity and the Rademacher averages of the trace classes $S_{p}$ ($1\leq p<\infty$)},
        url = {http://eudml.org/doc/217886},
        volume = {50},
        year = {1974},
}

\bib{TJ79}{article}{
   author={Tomczak-Jaegermann, N.},
   title={Computing $2$-summing norm with few vectors},
   journal={Ark. Mat.},
   volume={17},
   date={1979},
   number={2},
   pages={273--277},
   issn={0004-2080},
   %review={\MR{608320}},
   %doi={10.1007/BF02385473},
}


\bib{TJ89}{book}{
   author={Tomczak-Jaegermann, N.},
   title={Banach-Mazur distances and finite-dimensional operator ideals},
   series={Pitman Monographs and Surveys in Pure and Applied Mathematics},
   volume={38},
   publisher={Longman Scientific \& Technical, Harlow; copublished in the
   United States with John Wiley \& Sons, Inc., New York},
   date={1989},
   pages={xii+395},
   %isbn={0-582-01374-7},
   %review={\MR{993774}},
}

\end{thebibliography}
\end{document}